\documentclass[10pt]{article}%

\usepackage[T1]{fontenc}
\usepackage[latin9]{inputenc}
\usepackage{geometry}
\usepackage{amstext}
\usepackage{amssymb,amsmath}
\usepackage{esint}
\usepackage{amsthm}
\usepackage{color}
\usepackage{bm}
\usepackage{enumerate}
\usepackage{graphicx}
\usepackage{appendix}
\usepackage{underscore}

\usepackage{float} 
\usepackage{subfigure}

\usepackage{bbm}
\usepackage{xcolor}
\theoremstyle{plain}
\numberwithin{equation}{section}
\usepackage{indentfirst}
\newtheorem{theorem*}{Theorem}

\newtheorem{definition*}{Definition}
\newtheorem{lemma}{Lemma}

\newtheorem{proposition*}{Proposition}

\newcommand{\m}[1]{\ensuremath{\mathbf{#1}}}

\usepackage{hyperref}
\hypersetup{bookmarksopen=true}

\def\ra{\rangle} \def\la{\langle}

\usepackage[T1]{fontenc}

\usepackage{textcomp}

\usepackage{relsize}

\begin{document}

\title{Acoustic scattering and field enhancement \\through a subwavelength aperture}
\author{Ying Liang\footnote{Department of Mathematics, Purdue University, West Lafayette, Indiana 47907, USA. (liang402@purdue.edu).}
\and Jun Zou\footnote{Department of Mathematics, The Chinese University of Hong Kong, Shatin, N.T., Hong Kong. The work of this author was substantially supported by Hong Kong RGC General Research Fund (Project 14306718) and NSFC/Hong Kong RGC Joint Research Scheme 2016/17 (Project N\_CUHK437/16). (zou@math.cuhk.edu.hk).}}
\date{}
\maketitle\maketitle

\begin{abstract}
This work is concerned with the acoustic field enhancement through a circular hole perforated in a rigid plate of finite thickness. We construct the Green's functions for the subdomains in this structure and derive the boundary-integral equations to investigate the mechanism of the scattering and field enhancement. The asymptotic expansions of Fabry-Perot type resonances with respect to the size of the aperture are presented, and the quantitative analysis of the field enhancement at resonant frequencies is carried out with both the enhancement order and shapes of the resonant modes characterized. We also investigate the transmission in the nonresonant quasi-static regime, which suggests that the velocity field is enhanced due the the fast transition of pressure field across the aperture.
\end{abstract}

\indent{\footnotesize asymptotic analysis}

 {\footnotesize {\bf Mathematics Subject Classification(MSC2000)}: 35C20, 35J05, 35P30, 78A45}

%
%
%
%

\begin{section}{Introduction}
The studies of electromagnetic and acoustic scattering by subwavelength structures were extensively pursued in recent years for the possibilities of realizing  transmitted waves with novel wave characteristics utilizing the local resonances of subwavelength structures, see \cite{lin2017scattering, lin2018scattering,lezec2002beaming,martin2004optical,astilean2000light, holley2019extraordinary}, for instance.
The acoustic transmission phenomenon for the structured plates is analogous to the phenomenon of electromagnetic waves, which motivates studies exploring the transferability of extraordinary optical transmissions \cite{ebbesen1998extraordinary} to the acoustic case. In particular, extensive research effort has been devoted to the investigations on the transmission or diffraction by various acoustical gratings, e.g., one-dimensional periodic slits in a rigid screen \cite{miles2002resonant}, a one-dimensional grating of finite grating thickness on steel rods  \cite{zhang2005acoustic}, and a single hole in a thick plane \cite{seo2000acoustic}, etc. It was reported in \cite{hou2013enhanced, christensen2008theory} that the transmission peaks of the hole array in thick plates are related to the Fabry-Perot-like (FP-like) resonances inside the holes.

In many cases, the enhancement of the transmitted wave is very sensitive to the geometry and the wave frequency, which motivates us to study the phenomenon in a simple configuration, where an accurate description of the acoustic field can be obtained with analytical approaches. The object of study in our work is a subwavelength hole perforated in a rigid plate of finite thickness which generates the extraordinary enhancement of acoustic transmission and works as the building block of acoustic meta-material in various applications, see  \cite{wu2018perspective, yang2017sound,ammari2017sub, ammari2015mathematical}. 
In the literature, the scattering phenomenon in this configuration has been studied by physical experiments \cite{hou2013enhanced} and mode matching \cite{seo2000acoustic} previously. 
Our goal is to present a rigorous quantitative analysis of the field enhancement for the acoustic scattering and give a complete picture for the mechanism of such enhancement utilizing integral equations, asymptotic analysis and other techniques.


 In this work, we will investigate the mechanism of resonant scattering by the subwavelength hole perforated in a rigid plate as depicted in Figure \ref{cross}. Helmholtz equation is used as the mathematical model of the diffraction phenomenon of the subwavelength structure. 
  As studied in previous works (see, e.g., \cite{ammari2002analysis,Ori2020}), the Helmholtz equation with certain boundary conditions possesses a unique solution for wavenumber $k$ with $\text{Im}\, k>0$. Therefore, we can deduce from the analytic continuation that there exists a unique solution for all complex wavenumbers $k$ except for a countable number of points, which are poles of the resolvent associated with the transmission problem.
These poles are referred to as the  resonant frequencies (or simply resonances) of the scattering problems, and these associated nontrivial solutions are referred to as the resonant modes. 
If the wavenumber $k$ of the incident wave is close to the real part of a resonance and the imaginary part of the resonance is small, an enhancement of the field will be observed, as we will elaborate in Section~\ref{quantreso}.
 To this end, we utilize the analytical tools to justify the existence of the resonant frequencies and investigate the order of field enhancement of the resonant frequencies.

\begin{figure}
\centering{
\includegraphics[width=5in]{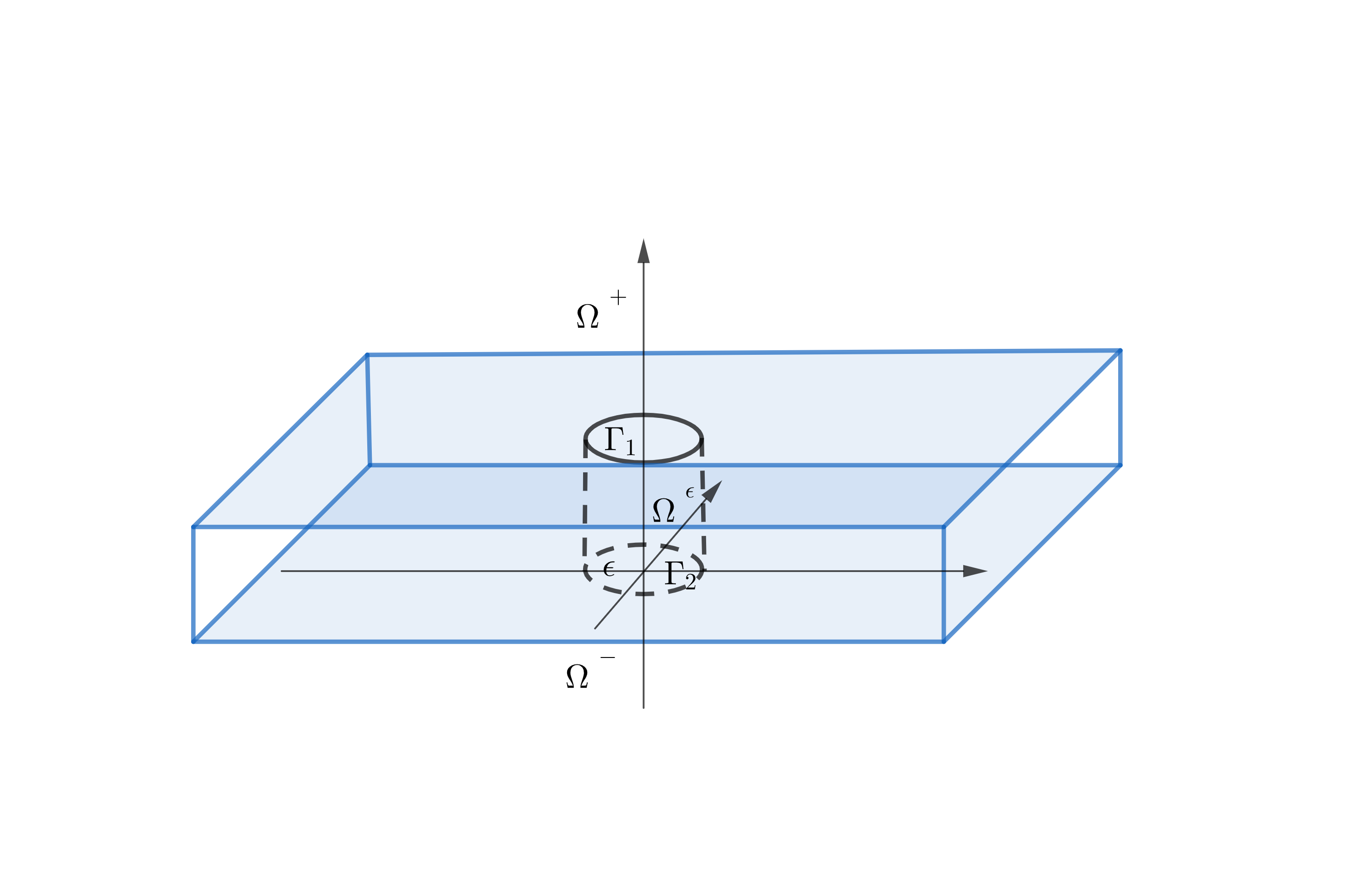}}
\caption{infinite plate of finite thickness with circular hole}\label{cross}
\end{figure}

Throughout this work, we will denote the radius of the hole by $\varepsilon$ and the thickness of the plate by $\ell$ as shown in Figure \ref{cross}, and we shall focus on the subwavelength regime, i.e., the radius $\varepsilon$ much smaller than the thickness of the board $\ell$ and the wavelength of the incident field $\lambda$.
 The analysis starts with the construction of a Green's function with the homogeneous Neumann boundary condition on each subdomain, and  the formulation of the boundary-integral equations equivalent to the original problem. 
 Asymptotic analysis of the integral equations will be carried out to demonstrate that the acoustic field enhancement through a subwavelength hole can be attributed to the resonances, that is, the transmission peaks at specific frequencies corresponding to particular resonances, referred to as Fabry-Perot type resonances. 
 The main contributions of this paper are proving rigorously the existence of Fabry-Perot type resonances, deriving the asymptotic expansions for those resonances, and further analyzing quantitatively the field enhancement near these resonant frequencies and also the asymptotic behavior of fields at nonresonant frequencies in the quasi-static regime. 
More precisely, it is proved that the enhancement of the pressure field  with an order of $O(\varepsilon^{-2})$ occurs at the Fabry-Perot type resonances, and the enhancement of the velocity field with an order of $O((k\ell)^{-1})$ occurs at the nonresonant frequencies in the quasi-static regime.

It should be mentioned that the mathematical studies on the acoustic impedance of perforated plate with a circular aperture have also been carried out previously in \cite{Ori2020}, where the matched asymptotic
expansion techniques are applied to derive the analytical formula of the acoustic impedance. On the other hand, the technique of asymptotic expansion has been applied to study the electromagnetic field enhancement of single slits in \cite{lin2017scattering, lin2018scattering}. We also refer to a related research of scattering by subwavelength cavities in \cite{bonnetier2010asymptotic}, where the layer potential techniques and Gohberg-Sigal theory are applied to study the resonances. For a comprehensive review of these techniques we refer the readers to \cite{ammari2009layer}.

The rest of the paper is organized as follows. In Section~\ref{problem}, we describe our model mathematically using the Helmholtz equation. In Section~\ref{inteq}, we introduce the Green's functions corresponding to the structure and derive the boundary-integral equations. In Section~\ref{asym} we carry out the asymptotic analysis of the integral equations and study the corresponding integral operators. We present the main results on the asymptotic expansions of resonant frequencies in Section~\ref{secresonance}. Section~\ref{quantreso} is devoted to the quantitative analysis on the enhancement of the acoustic fields at the resonant frequencies. Finally in Section~\ref{quantnon} we investigate the asymptotic behavior of scattered fields in the nonresonant quasi-static regime.

\end{section}
\begin{section}{Problem description}\label{problem}

 Throughout this work, we consider the time-harmonic acoustic scattering wave through a circular hole in an infinite rigid plate with finite thickness depicted in Figure \ref{cross}, and use the standard cylindrical coordinate system $\m{r} =(r,\theta, z)$.
 As mentioned in the Introduction, we use $k$ to denote the wavemunber, $\varepsilon$ to denote the radius of the hole, $\ell$ to denote the thickness of the plate, and $\Omega^\varepsilon$ to denote the cylindrical cavity, that is, $\Omega^\varepsilon: = \{\m{r}=(r,\theta, z):0\leq r< \varepsilon,-\pi< \theta\leq \pi, 0<z< \ell\}$.
 We further  denote the domain exterior to the perforated plate by $\hat{\Omega}^\varepsilon: =\Omega^+\cup \Omega^\varepsilon\cup\Omega^- $, where $\Omega^+ = \{\m{r}=(r,\theta, z):-\pi< \theta\leq \pi, z>\ell\}$ and $\Omega^-= \{\m{r}=(r,\theta, z): -\pi< \theta\leq \pi, z<0\}$ denote the semi-infinite domains above and below the plate, respectively.

 The total pressure field, denoted by $u_\varepsilon$, consists of three parts in the upper domain $\Omega^+$: the incident field $u^i$, the reflected field $u^r$, and the scattered field $u_\varepsilon^s$  radiating from the aperture $\Gamma_1: = \{\m{r}=(r,\theta,z):0\leq r\leq \varepsilon,-\pi< \theta\leq \pi, z=\ell\}$. 
 In the lower domain $\Omega^-$, the total pressure field only consists of the scattered field $u_\varepsilon^s$ radiating from the aperture $\Gamma_2: = \{\m{r}=(r,\theta,z):0\leq r\leq \varepsilon, -\pi< \theta\leq \pi, z=0\}$.
  We assume that the total field $u_\varepsilon$ satisfies the boundary condition $\frac{\partial u_\varepsilon}{\partial\nu}=0$ on the surface of this perforated plate (sound hard boundary condition), where $\nu$ is the outward normal vector pointing to $\hat{\Omega}^\varepsilon$. 
 In addition, at infinity, the scattered field $u_\varepsilon^s$ satisfies the Sommerfeld radiation condition \cite{colton2012inverse}.
 Now we are ready to formulate the diffraction problem of our interest as follow:
\begin{equation}
\begin{cases}
\Delta  u_\varepsilon+k^2 u_\varepsilon = 0  &\text{ in }  \hat{\Omega}^\varepsilon,\\
u_\varepsilon = u^i+u^r+u_\varepsilon^s&\text{ in } \Omega^+,\\
u = u_\varepsilon^s&\text{ in }\Omega^-,\\
\dfrac{\partial u_\varepsilon}{\partial\nu}=0&\text{ on } \partial \hat{\Omega}^\varepsilon,\\
\lim\limits_{|\m{r}|\to\infty} \sqrt{|\m{r}|}\left(\dfrac{\partial u_\varepsilon^s}{\partial |\m{r}|}-ik u_\varepsilon^s\right)  = 0.&
\end{cases}\label{inieq}
\end{equation}

In this work, the incident field is a time-harmonic acoustic plane wave given by  $u^i(\m{r}) = e^{ik (d_1 r\cos\theta -d_3(z-\ell))}$, with the direction unit vector $\m{d} = (d_1,0,-d_3)$, and the reflected field $u^r(\m{r}) = e^{ik (d_1 r\cos\theta+d_3(z-\ell))}$, with the direction unit vector $\m{d}' = (d_1,0, d_3)$. 
For clarity of exposition, we take the thickness of the plate $\ell=1$ in the following sections. If the thickness of the plate $\ell\neq 1$ and $\varepsilon\ll \ell$, the theoretical results  for $\ell=1$ can be generated directly by a scaling argument.

\end{section}
\begin{section}{Boundary-integral equations of the scattering problem}\label{inteq}
In this section, we derive the Green's functions of the Helmholtz equation \eqref{inieq} satisfying certain boundary conditions, and then reformulate \eqref{inieq} as a system of boundary-integral equations. 
  Let us first introduce the Green's function  $g^e(k;\m{r},\m{r}')$ which satisfies
\begin{eqnarray*}
\Delta  g^e(k;\m{r},\m{r}')+k^2 g^e(k;\m{r},\m{r}')=\delta(\m{r}-\m{r}'), \quad\m{r},\m{r}'\in \Omega^{\pm},
\end{eqnarray*}
with boundary conditions $\frac{\partial g^e(k;\m{r},\m{r}')}{\partial\nu_{\m{r}'}} = 0$ for $\m{r}'\in\partial\Omega^+$, $\m{r}\in\Omega^+$ or  $\m{r}'\in\partial\Omega^-$, $\m{r}\in\Omega^-$, where $\nu$ denotes the normal vector pointing to the
exterior domain $\Omega^{\pm}: = \Omega^+\cup\Omega^-$. One can easily check that the Green's function  $g^e(k;\m{r},\m{r}')$ takes the following form:
\begin{equation}
 g^e(k;\m{r},\m{r}') = \phi(k;|\m{r}-\m{r}'|)+\phi(k;|\bar{\m{r}}-\m{r}'|),\label{gex}
\end{equation}
where $\phi(k;|\m{r}-\m{r}'|)$ denotes the free space Green's function $-\frac{e^{ik|\m{r}-\m{r}'|}}{4\pi |\m{r}-\m{r}'|}$ in $\mathbb{R}^3$, and 
\begin{equation*}
\bar{\m{r}} = \begin{cases}
( r, \theta, 2-z)&\text{ if } \m{r},\m{r'}\in \Omega^+,\\
( r, \theta, -z)&\text{ if } \m{r},\m{r'}\in \Omega^-.
\end{cases}
\end{equation*}
Taking advantage of simplicity of the geometry, we will derive the Green's function corresponding to the cylindrical domain $\Omega^\varepsilon$ using eigenfunctions, present a decomposition of the Green's function for further calculation, and formulate the boundary-integral equations.


\begin{subsection}{ Green's function for the cylindrical cavity}
Now we construct the Green's function $g^i_\varepsilon(k;\m{r},\m{r}')$ satisfying
the homogeneous Neumann boundary conditions for $\Omega^\varepsilon$ as shown in Figure \ref{cross}. 
The Green's function $g^i_\varepsilon(k;\m{r},\m{r}')$ satisfies
\begin{equation}
\begin{cases}
\Delta  g^i_\varepsilon(k;\m{r},\m{r}')+k^2 g^i_\varepsilon(k;\m{r},\m{r}')= \delta(\m{r}-\m{r}'), \quad&\m{r},\m{r}'\in \Omega^{\varepsilon}, \\
\dfrac{\partial g^i_\varepsilon }{\partial\nu_{\m{r}'}} = 0,\quad &\m{r}'\in\partial\Omega^\varepsilon,
\end{cases}\label{cavity}
\end{equation}
where $\partial\Omega^{\varepsilon}$ consists of three parts, $\Gamma_1$, $\Gamma_2$ and $\Gamma_r = \{\m{r}=(r,\theta, z)\in \partial\Omega^{\varepsilon} : r = \varepsilon\}$.
 For integers $m\geq 1$, $n\geq 0$, and $j\geq 0$, we define the basis functions:
 \begin{equation}\label{eigenf}
 \phi_{mnje}(\m{r}) = J_n(\alpha_{mn} r)  \cos(n\theta)\cos(\gamma_j z), \quad  \phi_{mnjo}(\m{r}) = J_n(\alpha_{mn} r) \sin(n \theta) \cos(\gamma_j z),
\end{equation}
 where $\gamma_j =j\pi$, $J_n$ denotes the Bessel function of order $n$, $\alpha_{mn} = \frac{q_{mn}}{\varepsilon}$ and  $q_{mn}$ denotes the $m$-th smallest root of the equation
\begin{equation*}
\frac{\partial J_n(r)}{\partial r}=0.
\end{equation*}
It is noted that  $q_{1,0} = 0$, thus  $\alpha_{1,0} =0$.  
The basis functions $\{\phi_{mnje}, \phi_{mnjo}\}$ defined in \eqref{eigenf} are eigenfunctions of the Laplacian operator, and orthogonal to each other. More precisely, considering the Laplacian operator in cylindrical coordinates
$$\Delta \phi = \frac{1}{r} \frac{\partial}{\partial r}\left(r\frac{\partial \phi}{\partial r}\right)+\frac{1}{r^2}\frac{\partial^2 \phi}{\partial \theta^2}+\frac{\partial^2 \phi}{\partial z^2},$$
 as the Bessel function $J_n$ satisfies
${\displaystyle r^2{\frac {d^{2}J_n}{dr^{2}}}+r{\frac{d J_n}{dr}}+(r^{2}-n^{2})J_n=0},$
one can verify that
\begin{equation}
\Delta \phi_{mnje}(\m{r}) +\left((\alpha_{mn})^2+(\gamma_j)^2\right)\phi_{mnje}(\m{r}) = 0, \quad \Delta \phi_{mnjo}(\m{r}) +\left((\alpha_{mn})^2+(\gamma_j)^2\right)\phi_{mnjo}(\m{r}) = 0.\label{eigen1}
\end{equation}
Then the expansion of $g^i_\varepsilon(k;\m{r},\m{r}')$ takes the following form: 
\begin{equation}
g^i_\varepsilon(k;\m{r},\m{r}') = \sum_{m, n}\sum_{j=1}^\infty\sum_{\xi \in\{o,e\}} c_{mnj}d_{mnj} \phi_{mnj\xi}(\m{r})\phi_{mnj\xi}(\m{r}'),\label{cavityG}
\end{equation}
where the coefficients are given by
\begin{equation*}
d_{mnj} = \dfrac{1}{\int_{\Omega^\varepsilon} (J_n(\alpha_{mn} r))^2  \cos^2(n\theta) \cos^2(\gamma_j z)d\m{r}}= \dfrac{4}{\pi (1+\delta_{0, j})
(1+\delta_{0,n})
\int_0^\varepsilon  (J_n(\alpha_{mn} r))^2 r d r},
\end{equation*}
and $c_{mnj} = \frac{1}{k^2-(\alpha_{mn})^2-(\gamma_j)^2}$. Unless otherwise specified, the notation $\sum\limits_{m,n}$ denotes the summation over integers  $m\geq 1$, $n\geq 0$.
  We refer readers to \cite{tai1994dyadic, collin1990field} for this technique of constructing Green's functions. It is noted that the expansion \eqref{cavityG} is well defined on the whole complex $k$-plane except for a countable number of points $k =\sqrt{(\alpha_{mn})^2+(\gamma_j)^2}$, which are the eigenvalues of the Laplacian operator in the cavity $\Omega^\varepsilon$. We will take a limiting procedure to evaluate $g^i_\varepsilon(k;\m{r},\m{r}')$ when $k$ turns to those eigenvalues, and take the similar procedures in subsequent analysis. 

\end{subsection}

\begin{subsection}{Green's function for the semi-infinite waveguide}
\begin{figure}
\centering{
\includegraphics[width=5in]{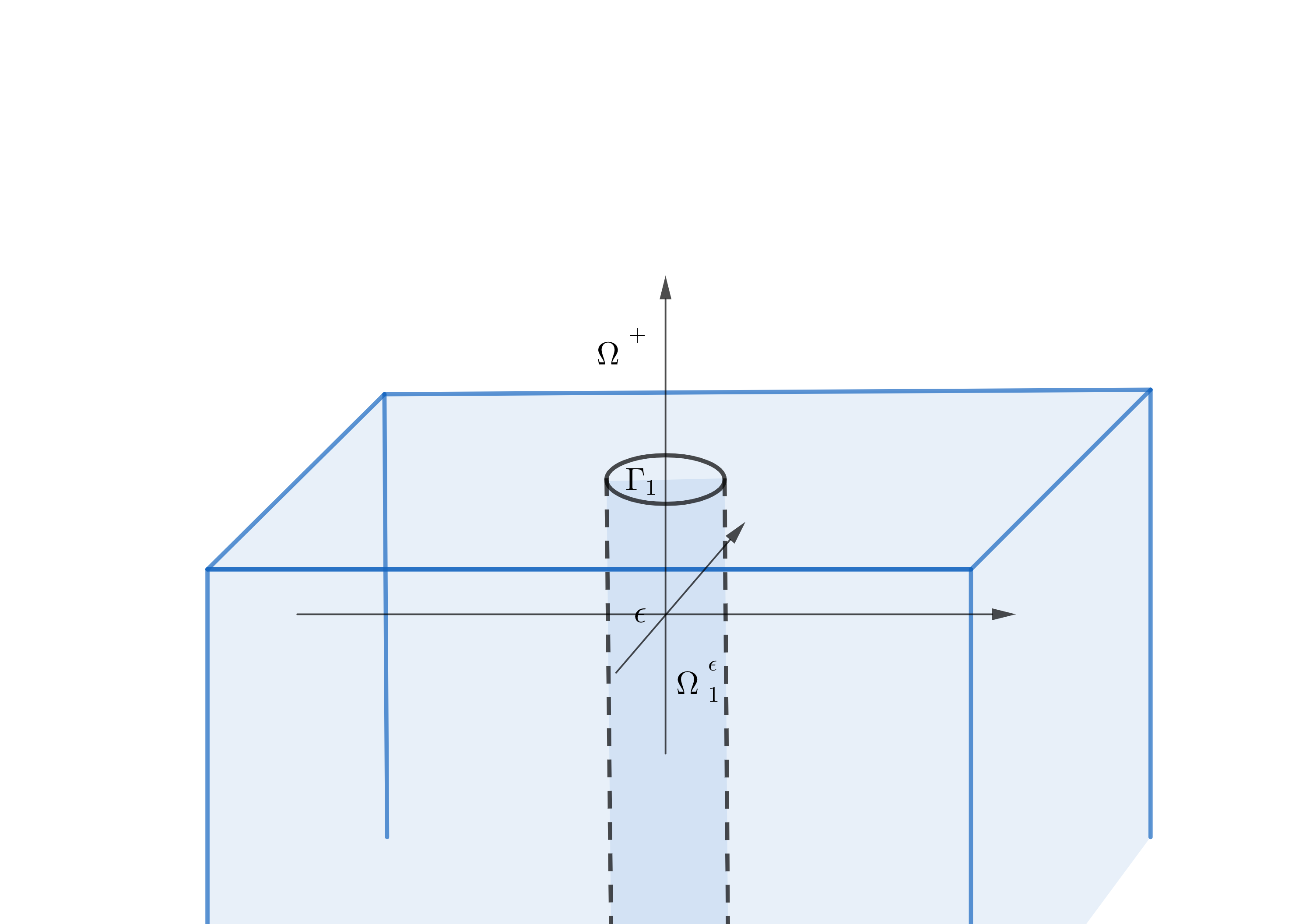}}
\caption{infinite plate of infinite thickness with circular hole}\label{cross2}
\end{figure}

To prepare for the decomposition of the Green's function $g^i_\varepsilon(k;\m{r},\m{r}')$ for $\Omega^\varepsilon$, we would like to introduce another  configuration consisting of a semi-infinite waveguide $\Omega_1^\varepsilon: = \{ \m{r} = (r,\theta, z): 0\leq r< \varepsilon,-\pi< \theta\leq \pi, z<1 \}$ and the upper semi-infinite domain $\Omega^+$ as shown in Figure \ref{cross2}. We consider the Green's function $g_0^1(k;\m{r},\m{r}')$ that satisfies the homogeneous Neumann boundary condition on $\partial\Omega^\varepsilon_1$, i.e.,
\begin{equation}
\begin{cases}
\Delta  g_0^1(k;\m{r},\m{r}')+k^2 g_0^1(k;\m{r},\m{r}')= \delta(\m{r}-\m{r}'), \quad &\m{r},\m{r}'\in \Omega^{\varepsilon}_1, \\
 \dfrac{\partial g_0^1 }{\partial\nu_{\m{r}'}} = 0,\quad &\m{r}'\in\partial\Omega^\varepsilon_1,
\end{cases}
\label{waveguide}
\end{equation}
with decaying boundary condition $\frac{\partial g_0^1}{\partial z}\to 0$ as $z\to -\infty$.
To construct $g_0^1(k;\m{r},\m{r}')$, we introduce the following set of basis functions:
\begin{equation}
\phi_{mne}(\m{r}) = J_n(\alpha_{mn} r)\cos(n\theta),\quad \phi_{mno}(\m{r}) = J_n(\alpha_{mn} r) \sin(n \theta),
 \label{eigen0}
\end{equation}
for integers $m\geq 1$ and $n\geq 0$.
The basis functions $\{\phi_{mne}, \phi_{mno}\}$ defined in \eqref{eigen0} are eigenfunctions of the Laplacian operator on $\Gamma_1$, and orthogonal to each other. Following the definition of the Bessel function $J_n$, one can verify that
\begin{equation}
\Delta \phi_{mne}(\m{r}) +(\alpha_{mn})^2\phi_{mne}(\m{r}) = 0, \quad \Delta \phi_{mno}(\m{r}) +(\alpha_{mn})^2\phi_{mno}(\m{r}) = 0.\label{eigen1}
\end{equation}
With the eigenfunctions $\{\phi_{mne}, \phi_{mno}\}$, we derive the expansion of the Green's function $g_0^1(k;\m{r},\m{r}')$ for the semi-infinite waveguide $\Omega_1^\varepsilon$ as 
\begin{equation}\label{g01}
g_0^1(k;\m{r},\m{r}') = \sum_{(m, n)\neq (1,0)}^\infty  \sum_{\xi \in\{o,e\}}c_{mn} d_{mn}\phi_{mn\xi}(\m{r})\phi_{mn\xi}(\m{r}')(e^{-\beta_{mn} |z-2+z'|}+e^{-\beta_{mn} |z-z'|}),
\end{equation}
where 
$$d_{mn} = \dfrac{1}{\int_0^\varepsilon \int_0^{2\pi} (J_n(\alpha_{mn} r))^2\cos^2(n\theta) rd\theta dr} = \dfrac{2}{\pi (1+\delta_{0, n})\int_0^\varepsilon  (J_n(\alpha_{mn} r))^2 r dr}, $$
$\beta_{mn} =\sqrt{(\alpha_{mn})^2-k^2}$, $c_{mn} =-\frac{1}{2\beta_{mn}}$, and $\beta_{mn}$ has positive real part.  Note that since $\frac{\partial g_0^1}{\partial z}\to 0$ as $z\to -\infty$, the  propagative eigenfunction $\phi_{1,0,\xi}(\m{r})$ shall be excluded and the eigenfunctions are summed over integers $m \geq 1$, $n\geq 0$ and $(m,n)\neq (1,0)$ in \eqref{g01}.  

\end{subsection}

\begin{subsection}{Decomposition of the Green's function for cylindrical cavity}

To prepare for the asymptotic analysis of the integral operators in following sections, we will introduce a useful decomposition of the Green's function $g^i_\varepsilon(k;\m{r},\m{r}')$ to separate the effects of the Green's function due to different components of the geometry. Comparing the geometries of $\Omega^\varepsilon$ and $\Omega_1^\varepsilon$, we deduce that the corresponding Green's function $g^i_\varepsilon(k;\m{r},\m{r}')$ of $\Omega^\varepsilon$ can be decomposed into three parts:
\begin{equation}\label{decomg}
 g^i_\varepsilon(k;\m{r},\m{r}')= g_0^1(k;\m{r},\m{r}')+g_0^p(k;\m{r},\m{r}')+g_0^2(k;\m{r},\m{r}'),
 \end{equation} 
   where $g_0^1(k;\m{r},\m{r}')$ is the Green's function of  $\Omega_1^\varepsilon$,  $g_0^p(k;\m{r},\m{r}')$  contains the propagative eigenfunction $\phi_{1,0,e}(\m{r})$ along the $z$ direction, and $g_0^2(k;\m{r},\m{r}')$ takes care of the homogeneous  Neumann boundary condition on $\Gamma_2$. The eigenfunction $\phi_{1,0,e}(\m{r})$ is excluded when deriving $g_0^1(k;\m{r},\m{r}')$ for $\Omega_1^\varepsilon$, while it shall not be neglected when constructing the Green's function $g^i_\varepsilon(k;\m{r},\m{r}')$ for the cavity $\Omega^\varepsilon$. We first present the formula of $g_0^p(k;\m{r},\m{r}')$ as 
 \begin{equation*}
g_0^p(k;\m{r},\m{r}') = c_{1,0}d_{1,0} \phi_{1,0,e}(\m{r})\phi_{1,0,e}(\m{r}') (e^{-\beta_{1,0}|z-2+z'|}+e^{-\beta_{1,0}|z-z'|}),
\end{equation*}
where we take $\beta_{1,0} = ik$, $c_{1,0} =-\frac{1}{2\beta_{1,0}}$ and $d_{1,0}= \frac{1}{\pi\varepsilon^{2}}$ to normalize this term. The third part  $g_0^2(k;\m{r},\m{r}')$ shall satisfy the following system of equations to ensure the homogeneous  Neumann boundary condition on $\Gamma_2$ of the summation \eqref{decomg}:
\begin{equation*}
\begin{cases}
\Delta g_0^2(k;\m{r},\m{r}') +k^2 g_0^2(k;\m{r},\m{r}') = 0,  \quad &\m{r},\m{r}'\in \Omega^\varepsilon,\\
\dfrac{\partial g_0^2}{\partial\nu_{\m{r}'}} = 0, \quad &\m{r}'\in \Gamma_1\cup\Gamma_r,\\
\dfrac{\partial g_0^2}{\partial\nu_{\m{r}'}} = -\dfrac{\partial g_0^1}{\partial\nu_{\m{r}'}} -\dfrac{\partial g_0^p}{\partial\nu_{\m{r}'}},  \quad &\m{r}'\in \Gamma_2.
\end{cases}
\end{equation*}
Using $w(k;\m{r},\m{r}')$ to denote $-\frac{\partial g_0^2}{\partial\nu_{\m{r}'}}|_{z'=0}= (\frac{\partial g_0^1}{\partial\nu_{\m{r}'}} +\frac{\partial g_0^p}{\partial\nu_{\m{r}'}} )|_{z'=0}$ and $E_{mn}(z)$ to denote $e^{-\beta_{mn} |z-2|}+e^{-\beta_{mn} |z|}$, then $w(k;\m{r},\m{r}') $ takes the following form:
\begin{equation*}
\begin{aligned}
w(k;\m{r},\m{r}') &=  -\sum_{m, n} \sum_{\xi \in\{o,e\}}c_{mn}d_{mn} \phi_{mn\xi}(\m{r})\phi_{mn\xi}(\m{r}') \beta_{mn}\left(e^{-\beta_{mn} |z-2|}+e^{-\beta_{mn} |z|}\right) \\
&:=-\sum_{m, n}\sum_{\xi \in\{o,e\}} c_{mn}d_{mn} \phi_{mn\xi}(\m{r})\phi_{mn\xi}(\m{r}')\beta_{mn}E_{mn}(z).
\end{aligned}
\end{equation*}
Applying the formula \eqref{cavityG} of the Green's function $g^i_\varepsilon(k; \tilde{\m{r}},\m{r}')$, we derive the integral representation of $g_0^2(k; \tilde{\m{r}},\m{r}')$ for $\tilde{\m{r}}, \m{r}'\in\Omega^\varepsilon$:
\begin{equation*}
\begin{aligned}
g_0^2(k; \tilde{\m{r}},\m{r}') &= \int_{\Gamma_2} g^i_\varepsilon(k;\m{r}',\m{r}'') w(k;\tilde{\m{r}}, \m{r}'') ds(\m{r}'') \\
&=-\int_{\Gamma_2}  \left(\sum_{m, n}\sum_{j=0}^\infty\sum_{\xi \in\{o,e\}} c_{mnj}d_{mnj} \phi_{mnj\xi}(\m{r}')\phi_{mnj\xi}(\m{r}'') \right)\\
&\quad \cdot\left(\sum_{m, n} \sum_{\xi \in\{o,e\}}c_{mn}d_{mn} \phi_{mn\xi}(\tilde{\m{r}})\phi_{mn\xi}(\m{r}'') \beta_{mn}E_{mn}(\tilde{z})\right)
 ds(\m{r}'')  \\
 &=- \sum_{m, n}\sum_{\xi \in\{o,e\}}\sum_{j=0}^\infty\left( c_{mnj}d_{mnj} \phi_{mnj\xi}(\m{r}')\right)\left( c_{mn}d_{mn} \phi_{mn\xi}(\tilde{\m{r}}) \beta_{mn}E_{mn}(\tilde{z})\right)\int_{\Gamma_2} (\phi_{mn\xi}(\m{r}''))^2  ds(\m{r}'')  \\
  &= \sum_{m, n}\sum_{j=0}^\infty \sum_{\xi \in\{o,e\}}\frac{1}{1+\delta_{0, j}}\left( \frac{  d_{mnj}E_{mn}(\tilde{z})}{k^2-(\alpha_{mn})^2-(\gamma_j)^2} \phi_{mnj\xi}(\m{r}')\right)\phi_{mn\xi}(\tilde{\m{r}}),
  \end{aligned}
  \end{equation*}
where we  have used $c_{mn}\beta_{mn} = -\frac{1}{2}$. Thus we are ready to present the expansion of the Green's function $ g^i_\varepsilon(k;\m{r},\m{r}')$ through the decomposition \eqref{decomg}.
 For $\m{r}, \m{r}'\in\Gamma_1$, $g^i_\varepsilon(k;\m{r},\m{r}')$ can be written as the summation of the following three terms:
\begin{align}
 g_0^1(k; \m{r},\m{r}') &=  2\sum_{(m,n)\neq (1,0)} \sum_{\xi \in\{o,e\}}c_{mn}d_{mn} \phi_{mn\xi}(\m{r})\phi_{mn\xi}(\m{r}'), \label{sameg1} \\
g^p(k; \m{r},\m{r}') &=  2 c_{1,0}d_{1,0} \phi_{1,0,e}(\m{r})\phi_{1,0,e}(\m{r}'), \label{samegp}   \\
 g_0^2(k; \m{r},\m{r}')&= \sum^\infty_{m, n}\sum_{\xi \in\{o,e\}}\sum_{j=0}^\infty \frac{1}{1+\delta_{0, j}} \frac{(-1)^j E_{mn}(1)d_{mn}}{k^2-(\alpha_{mn})^2-(\gamma_j)^2} \phi_{mn\xi}(\m{r}') \phi_{mn\xi}(\m{r})\nonumber \\
 & =\sum^\infty_{m, n}\sum_{\xi \in\{o,e\}}\sum_{j=0}^\infty \frac{2}{1+\delta_{0, j}} \frac{(-1)^j e^{-\beta_{mn}}d_{mn}}{k^2-(\alpha_{mn})^2-(\gamma_j)^2} \phi_{mn\xi}(\m{r}') \phi_{mn\xi}(\m{r}). \label{sameg2} 
 \end{align}

\end{subsection}
\begin{subsection}{Boundary-integral equations for the diffraction}\label{subint}

We are now ready to formulate the boundary-integral equations of the transmission across the aperture. Note that $\frac{\partial u^i}{\partial\nu}+\frac{\partial u^r}{\partial\nu}=0$ on $\{z= 1\}$, and then $\frac{\partial u_\varepsilon^s}{\partial\nu}=0$ on $\{z=1\}\backslash \Gamma_1$. Following the Green's identity, with the radiation condition in the far field, we obtain that the total pressure field $u_\varepsilon$ in the upper domain $\Omega^+$ satisfies:
\begin{equation*}
u_\varepsilon(\m{r}) = \int_{\Gamma_1} g^e(k;\m{r},\m{r}')\frac{\partial u_\varepsilon^s}{\partial\nu}(\m{r}')ds(\m{r}')+ u^i(\m{r})+u^r(\m{r}),\quad \m{r}\in\Omega^+,
\end{equation*} 
 where $\nu$ denotes the unit outward normal pointing to the
domain $ \Omega^+$. 
By the continuity of the single layer potential \cite{kress1989linear}, we further deduce:
\begin{equation}\label{up1}
u_\varepsilon(\m{r}) = \int_{\Gamma_1} g^e(k;\m{r},\m{r}')\frac{\partial u_\varepsilon^s}{\partial\nu}(\m{r}')ds(\m{r}')+ u^i(\m{r})+u^r(\m{r}),\quad \m{r}\in\Gamma_1.
\end{equation} 
Similarly, taking limit from the lower domain $\Omega^-$ leads to
\begin{equation}\label{low1}
u_\varepsilon(\m{r}) = \int_{\Gamma_2} g^e(k;\m{r},\m{r}')\frac{\partial u_\varepsilon^s}{\partial\nu}(\m{r}')ds(\m{r}'),\quad \m{r}\in\Gamma_2,
\end{equation} 
 where $\nu$ denotes the unit outward normal pointing to the
domain $ \Omega^-$.
Applying the Green's function  $g^i_\varepsilon(k;\m{r},\m{r}')$ and the boundary condition on $\partial \Omega^\varepsilon$, we derive the integral representation for the total field in $\Omega^\varepsilon$:
\begin{equation*}
u_\varepsilon(\m{r}) = -\int_{\Gamma_1\cup\Gamma_2} g^i_\varepsilon(k;\m{r},\m{r}')\frac{\partial u_\varepsilon^s}{\partial\nu}(\m{r}')ds(\m{r}'),\quad  \m{r}\in\Omega^\varepsilon.
\end{equation*} 
By the continuity of the single layer potential again, one has 
\begin{equation}\label{in1}
u_\varepsilon(\m{r}) = -\int_{\Gamma_1\cup\Gamma_2} g^i_\varepsilon(k;\m{r},\m{r}')\frac{\partial u_\varepsilon^s}{\partial\nu}(\m{r}')ds(\m{r}'),\quad  \m{r}\in\Gamma_1 \cup\Gamma_2.
\end{equation} 
%
Note that in these integral representations above, by the definition of $\nu$, 
\begin{equation*}
\left.\frac{\partial u_\varepsilon^s}{\partial\nu}\right\vert_{\Gamma_1} = \frac{\partial u_\varepsilon^s}{\partial z}(r,\theta,1), \quad\left. \frac{\partial u_\varepsilon^s}{\partial\nu}\right\vert_{\Gamma_2} = -\frac{\partial u_\varepsilon^s}{\partial z}(r,\theta,0).
\end{equation*}
Therefore, by imposing the continuity of the solution along  $\Gamma_1$ and $\Gamma_2$, we obtain the following system of boundary-integral equations for $\frac{\partial u_\varepsilon^s}{\partial\nu}|_{\Gamma_1}$ and $ \frac{\partial u_\varepsilon^s}{\partial\nu}|_{\Gamma_2}$  with representations \eqref{up1}, \eqref{low1}, and \eqref{in1}:
\begin{eqnarray}
\int_{\Gamma_1}g^e(k;\m{r},\m{r}')\frac{\partial u_\varepsilon^s}{\partial\nu}(\m{r}')ds(\m{r}')+ \int_{\Gamma_1\cup\Gamma_2}g^i_\varepsilon(k;\m{r},\m{r}')\frac{\partial u_\varepsilon^s}{\partial\nu}(\m{r}')ds(\m{r}')+  u^i(\m{r})+u^r(\m{r}) =0,
\quad &\m{r} \in \Gamma_1, \label{upper}\\
\int_{\Gamma_2}g^e(k;\m{r},\m{r}')\frac{\partial u_\varepsilon^s}{\partial\nu}(\m{r}')ds(\m{r}')+ \int_{\Gamma_1\cup\Gamma_2}g^i_\varepsilon(k;\m{r},\m{r}')\frac{\partial u_\varepsilon^s}{\partial\nu}(\m{r}')ds(\m{r}')=0, \quad&
\m{r}\in \Gamma_2\label{lower}.
\end{eqnarray}
\end{subsection}
\end{section}
\begin{section}{Asymptotic expansions of boundary-integral equations}\label{asym}
This section is devoted to the asymptotic expansions of boundary-integral equations \eqref{upper}--\eqref{lower} for the acoustic diffraction. 
To prepare for the asymptotic analysis, we shall rescale the cylindrical  coordinate system by introducing $\m{X} = (R,\theta) =(\frac{r}{\varepsilon}, \theta) \in D(0,1)$ for $\m{r}=(r,\theta, z)\in\Omega^\varepsilon$, where $D(0,1)$ is the unit disk in $\mathbb{R}^2$ centered at the origin, and then study the scaled operators defined on $D(0,1)$.
To facilitate the analysis of integral operators on $D(0,1)$, we first introduce the conventional fractional Sobolev space $H^{\frac{1}{2}}(\mathbb{R}^2)$ with the norm
\begin{equation*}
\Vert u\Vert_{H^{\frac{1}{2}}(\mathbb{R}^2)} = \int_{\mathbb{R}^2} (1+|\omega|^2)^{1/2}|\hat{u}(\omega)|^2d\omega,
\end{equation*}
 where $\hat{u}$ is the Fourier transform of $u$.
 Then we introduce two related fractional order spaces  $$V_1 = \tilde{H}^{-\frac{1}{2}}(D(0,1)), \quad V_2=H^{\frac{1}{2}}(D(0,1)),$$
 where $H^{\frac{1}{2}}(D(0,1))$ is the Hilbert space
 \begin{equation*}
H^{\frac{1}{2}}(D(0,1)):= \{u = U|_{D(0,1)}: U\in H^{\frac{1}{2}}(\mathbb{R}^2)\}
 \end{equation*}
  with the norm
  \begin{equation*}
\Vert u\Vert^2_{H^{\frac{1}{2}}(D(0,1))} = \inf \{ \Vert U\Vert_{H^{\frac{1}{2}}(\mathbb{R}^2)}: U\in H^{\frac{1}{2}}(\mathbb{R}^2) \text{ and } U|_{D(0,1)}=u\},
\end{equation*}
and $V_1$ is the dual of $V_2$.  We refer
to \cite{adams2003sobolev} for further information on the fractional Sobolev spaces. Through out this work,  the duality between $V_1$ and $V_2$ will be denoted as $\langle u,v \rangle$ for $u\in V_1$, $v\in V_2$. To simplify the notations, we introduce the expression $x \lesssim y$, which means $x \leq C y$ for
some generic constant $C$. If $x \lesssim y$ and $y \lesssim x$ hold simultaneously, then we write $x \sim y$.

\begin{subsection}{Asymptotic expansions of integral operators}
Let us introduce the following useful notations for the scaled quantities of $\m{X},\m{Y}\in D(0,1)$:
\begin{equation*}
\begin{aligned}
\Psi_1 (\m{X}) &:= -\frac{\partial u_\varepsilon^s}{\partial\nu}(\varepsilon\m{X}, 1),\\
\Psi_2 (\m{X}) &:= - \frac{\partial u_\varepsilon^s}{\partial\nu}(\varepsilon\m{X}, 0),\\
f(\m{X})&:=  (u^i+u^r)(\varepsilon\m{X}, 1),\\
G_\varepsilon^e(k;\m{X},\m{Y}) &:= 2\phi(k;|\m{X}-\m{Y}|\varepsilon),\\
G_\varepsilon^i(k;\m{X},\m{Y}) &:=g^i_\varepsilon(k;(\varepsilon \m{X},1),(\varepsilon \m{Y},1) )= g^i_\varepsilon(k;(\varepsilon \m{X},0),(\varepsilon \m{Y},0) ) ,\\
\tilde{G}_\varepsilon^i(k;\m{X},\m{Y}) &:=g^i_\varepsilon(k;(\varepsilon \m{X},1),(\varepsilon \m{Y},0) )= g^i_\varepsilon(k;(\varepsilon \m{X},0),(\varepsilon \m{Y},1)),
\end{aligned}
\end{equation*}
Applying the notations above, we define the following scaled integral operators of $\psi\in V_1$ :
\begin{eqnarray*}
T^e \psi(\m{X}) = \int_{D(0,1)} G_\varepsilon^e(k;\m{X},\m{Y}) \psi(\m{Y})d\m{Y},\quad \m{X}\in D(0,1),\\
T^i \psi (\m{X}) =  \int_{D(0,1)} G_\varepsilon^i(k;\m{X},\m{Y}) \psi(\m{Y})d\m{Y},\quad \m{X}\in D(0,1),\\
\tilde{T}^i \psi (\m{X}) =  \int_{D(0,1)} \tilde{G}_\varepsilon^i(k;\m{X},\m{Y}) \psi(\m{Y})d\m{Y},\quad \m{X}\in D(0,1).
\end{eqnarray*}
By the change of variable, we deduce the following proposition:
\begin{proposition*}
The system of boundary-integral equations \eqref{upper}--\eqref{lower} is equivalent to 
\begin{equation}
\begin{bmatrix}
T^e+ T^i& \tilde{T}^i\\
\tilde{T}^i& T^e+ T^i
\end{bmatrix}
\begin{bmatrix}
\Psi_1\\
\Psi_2
\end{bmatrix}
=\begin{bmatrix}
 \varepsilon^{-2} f\\
0
\end{bmatrix}.\label{initialmatrix}
\end{equation}
\end{proposition*}

Next we will carry out the asymptotic expansions of the integral operators in \eqref{initialmatrix}.  For clarity, we introduce the notations below for the asymptotic expansions:
\begin{align}
 \beta_1(k,\varepsilon) &:= -\frac{ik}{2\pi},\label{notbeta1}\\
  \beta_2(k,\varepsilon) &:= \frac{\text{cot}\,k}{k}d_{1,0},\label{notbeta2}\\
\beta(k,\varepsilon) &:= \beta_1(k,\varepsilon)+\beta_2(k,\varepsilon)= -\frac{ik}{2\pi}+\frac{\text{cot}\, k}{k}d_{1,0},\label{notbeta}\\
   \tilde{\beta}(k,\varepsilon) &:= \frac{ d_{1,0} }{k\sin k},\label{nottbeta}\\
   r^i(k;\m{X},\m{Y}) &:= g_0^1(k;(\varepsilon\m{X},1),(\varepsilon\m{Y},1)),\nonumber\\
   \kappa(k;\m{X},\m{Y}) &:=  -\frac{1}{2\pi |\m{X}-\m{Y}|\varepsilon} + r^i(k;\m{X},\m{Y}).\nonumber
\end{align}
Since the basis functions $\{\phi_{mn\xi}(\m{r})\}$ are invariant along the $z$ axis by definition, we suppress the dependence on $z$ and denote $\phi_{mn\xi}(\m{r})$ by  $\phi_{mn\xi}(\varepsilon \m{X})$  for $\m{r} = (\varepsilon\m{X},z)$ in the remaining.
  Now we present the asymptotic expansions of the kernels $G_\varepsilon^e(k;\m{X},\m{Y})$, $G_\varepsilon^i(k;\m{X},\m{Y})$ and $\tilde{G}_\varepsilon^i(k;\m{X},\m{Y})$ for the integral operators defined on $V_1$.
\begin{lemma}\label{LM1}
If $|k\varepsilon| \ll 1$, we have the following expansions:
\begin{align*}
G_\varepsilon^e(k;\m{X},\m{Y})&= \beta_1(k,\varepsilon)- \frac{1}{2\pi |\m{X}-\m{Y}|\varepsilon} +r^\varepsilon_1(\m{X},\m{Y}),\\
G_\varepsilon^i(k;\m{X},\m{Y})&=\beta_2(k,\varepsilon)+ r^i(k;\m{X},\m{Y})+ r^\varepsilon_2(\m{X},\m{Y}),\\
\tilde{G}_\varepsilon^i(k;\m{X},\m{Y})&=\tilde{\beta}(k,\varepsilon)+\tilde{\kappa}_\infty(\m{X},\m{Y}),
\end{align*}
where $r_1^\varepsilon(\m{X},\m{Y})$, $r_2^\varepsilon(\m{X},\m{Y})$ and $\tilde{\kappa}_\infty(\m{X},\m{Y})$ are bounded functions with $r_1^\varepsilon\sim O(k^2 \varepsilon)$, $r_2^\varepsilon\sim O( \varepsilon^{-1}\exp(-\varepsilon^{-1}))$, and $\tilde{\kappa}_\infty\sim O(\varepsilon^{-1}\exp(-\varepsilon^{-1}))$.
\end{lemma}
\begin{proof}
First we consider the expansion of $G_\varepsilon^e(k;\m{X},\m{Y})$. Direct Taylor's expansion leads to 
\begin{equation*}
\begin{aligned}
G_\varepsilon^e(k;\m{X},\m{Y}) &= -\frac{e^{ik\varepsilon|\m{X}-
\m{Y}|}}{2\pi|\m{X}-\m{Y}|\varepsilon}  \\
&= -\frac{i k}{2\pi} - \frac{1}{2\pi|\m{X}-\m{Y}|\varepsilon} +\frac{k^2 |\m{X}-\m{Y}|\varepsilon}{2\pi}  +o(k^2\varepsilon |\m{X}-\m{Y}|)\\
&:=-\frac{i k}{2\pi} - \frac{1}{2\pi |\m{X}-\m{Y}|\varepsilon} +r^\varepsilon_1(\m{X},\m{Y}).
\end{aligned}
\end{equation*}
 Then we consider the expansion of the function $G_\varepsilon^i(k;\m{X},\m{Y})$ applying the representation \eqref{sameg1}--\eqref{sameg2}:
\begin{align}
G_\varepsilon^i(k;\m{X},\m{Y}) &= g_0^1(k;(\varepsilon\m{X},1),(\varepsilon\m{Y},1))+g^p(k;(\varepsilon\m{X},1),(\varepsilon\m{Y},1))+g_0^2(k;(\varepsilon\m{X},1),(\varepsilon\m{Y},1))  \nonumber\\
&=  2\sum_{m, n}\sum_{\xi \in\{o,e\}} c_{mn}d_{mn} \phi_{mn\xi}(\varepsilon\m{X})\phi_{mn\xi}(\varepsilon\m{Y}) \nonumber\\
&+\sum_{\xi \in\{o,e\}}\sum_{j=0}^\infty  \frac{2}{1+\delta_{0, j}}\frac{(-1)^j e^{-\beta_{1,0}} d_{1,0}}{k^2-(\gamma_j)^2} \phi_{1,0,\xi}(\varepsilon\m{Y}) \phi_{1,0,\xi}(\varepsilon\m{X})\nonumber  \\
&+\frac{1}{2}\sum_{(m, n)\neq(1,0)}\sum_{j=0}^\infty\sum_{\xi \in\{o,e\}} \frac{2}{1+\delta_{0, j}} \frac{2(-1)^j e^{-\beta_{mn}} d_{mn}}{k^2-(\alpha_{mn})^2-(\gamma_j)^2} \phi_{mn\xi}(\varepsilon\m{Y}) \phi_{mn\xi}(\varepsilon\m{X})\nonumber \\
&= 2\sum_{(m, n)\neq (1,0)}\sum_{\xi \in\{o,e\}} c_{mn} d_{mn} \phi_{mn\xi}(\varepsilon\m{X})\phi_{mn\xi}(\varepsilon\m{Y})\nonumber \\
&+2c_{1,0}d_{1,0} \phi_{1,0,e}(\varepsilon\m{X})\phi_{1,0,e}(\varepsilon\m{Y})+ \frac{e^{-\beta_{1,0}} }{k \sin k}d_{1,0}\nonumber\\
&+\sum^\infty_{(m, n)\neq (1,0)}\sum_{\xi \in\{o,e\}}\frac{e^{-\beta_{mn}} d_{mn}\phi_{mn\xi}(\varepsilon\m{Y}) \phi_{mn\xi}(\varepsilon\m{X})}{\sqrt{-k^2+(\alpha_{mn})^2}\sinh(\sqrt{-k^2+(\alpha_{mn})^2})},\label{Gini}
\end{align}
and the  formula \eqref{Gini} can be further simplified as 
\begin{align}
G_\varepsilon^i(k;\m{X},\m{Y})&= 2\sum_{(m, n)\neq (1,0)}\sum_{\xi \in\{o,e\}}c_{mn} d_{mn} \phi_{mn\xi}(\varepsilon\m{X})\phi_{mn\xi}(\varepsilon\m{Y}) + \frac{\text{cot}\,  k  }{k }d_{1,0}\nonumber\\
&+\sum^\infty_{(m, n)\neq (1,0)}\sum_{\xi \in\{o,e\}}\frac{e^{-\beta_{mn}} d_{mn}\phi_{mn\xi}(\varepsilon\m{Y}) \phi_{mn\xi}(\varepsilon\m{X})}{\sqrt{-k^2+(\alpha_{mn})^2}\sinh(\sqrt{-k^2+(\alpha_{mn})^2})},
\end{align}
where we have used  $\phi_{1,0,e}(\varepsilon\m{X}) = J_0(0) = 1$,  $c_{1,0} = -\frac{1}{2\beta_{1,0}}$,  and the representations of elementary functions by series (cf. \cite{gradshteyn2014table}):
\begin{align*}
\sum_{j=1}^\infty \frac{2(-1)^j}{k^2-(j\pi)^2}+\frac{1}{k^2} &= \frac{1}{k \sin k},\\
 \sum_{j=1}^\infty\frac{2(-1)^j}{k^2-(\alpha_{mn})^2-(j\pi)^2}+\frac{1}{k^2-(\alpha_{mn})^2}
&=\frac{1}{\sqrt{-k^2+(\alpha_{mn})^2}\sinh(\sqrt{-k^2+(\alpha_{mn})^2})}.
\end{align*}
Recalling $\alpha_{mn} = \frac{q_{mn}}{\varepsilon}$, we obtain that the expansion of $G_\varepsilon^i(k;\m{X},\m{Y})$ is of the following form
\begin{equation*}
G_\varepsilon^i(k;\m{X},\m{Y})= \beta_2(k,\varepsilon)+ r^i(k;\m{X},\m{Y})+ r^\varepsilon_2(\m{X},\m{Y}),
\end{equation*}
where 
\begin{equation*}
\beta_2(k,\varepsilon) = \frac{\text{cot}\,k}{k}d_{1,0},
\end{equation*}
the second term 
\begin{equation*}
r^i(k;\m{X},\m{Y}) =g_0^1(k;(\varepsilon\m{X},1),(\varepsilon\m{Y},1)),
\end{equation*}
and the last term
\begin{equation*}
r_2^\varepsilon(\m{X},\m{Y}) = \sum^\infty_{(m, n)\neq (1,0)}\sum_{\xi \in\{o,e\}}\left(\frac{1}{\sqrt{-k^2+(\frac{q_{mn}}{\varepsilon})^2}\sinh(\sqrt{-k^2+(\frac{q_{mn}}{\varepsilon})^2})}\right)e^{-\beta_{mn}} d_{mn}\phi_{mn\xi}(\varepsilon\m{Y}) \phi_{mn\xi}(\varepsilon\m{X}).
\end{equation*}
We further deduce that $r_2^\varepsilon(\m{X},\m{Y})$ is of order $O( \varepsilon^{-1}\exp(-\varepsilon^{-1}))$ from $d_{mn} = O(\varepsilon^{-2})$ and $\beta_{mn} = \sqrt{(\frac{q_{mn}}{\varepsilon})^2-k^2} =O(\varepsilon^{-1})$.

Now we consider the expansion of the kernel $\tilde{G}_\varepsilon^i(k;\m{X},\m{Y})$. We use the formula \eqref{cavityG} of $g^i_\varepsilon(k;\m{r},\m{r}')$ directly:
 for $\m{r}\in \Gamma_1$ and $\m{r}'\in\Gamma_2$, we have the expansion representation 
 \begin{equation}
\tilde{G}_\varepsilon^i(k;\m{X},\m{Y}) =\sum_{m, n}\sum_{\xi \in\{o,e\}} \sum_{j=1}^\infty\left((-1)^j \frac{2c_{mnj}}{1+\delta_{0, j}} \right) d_{mn}\phi_{mn\xi}(\varepsilon\m{X})\phi_{mn\xi}(\varepsilon\m{Y}).\label{Gikernel}
\end{equation}
Again it follows from the representation of elementary functions as series that
\begin{equation*}
 \sum_j(-1)^j \frac{2c_{1,0,j}}{1+\delta_{0, j}}  = \frac{1}{k\sin k},
 \end{equation*}
 and for $(m,n)\neq (1,0)$,
 \begin{align*}
 \sum_j(-1)^j \frac{2c_{mnj}}{1+\delta_{0, j}}&=\frac{1}{\sqrt{-k^2+(\frac{q_{mn}}{\varepsilon})^2}\sinh(\sqrt{-k^2+(\frac{q_{mn}}{\varepsilon})^2})}\\
&=O\left(\frac{\varepsilon}{q_{mn}}\exp\left(-\frac{q_{mn}}{\varepsilon}\right)\right).
\end{align*}
One can now deduce from \eqref{Gikernel} that 
\begin{align}
\tilde{G}_\varepsilon^i(k;\m{X},\m{Y}) &=\frac{d_{1,0}}{k\sin k}+ \sum_{m, n} \sum_{\xi \in\{o,e\}}d_{mn}\phi_{mn\xi}(\varepsilon\m{X})\phi_{mn\xi}(\varepsilon\m{Y}) O\left(\frac{\varepsilon}{q_{mn}}\exp\left(-\frac{q_{mn}}{\varepsilon}\right)\right)  \nonumber \\
&:=\tilde{\beta}(k,\varepsilon)+\tilde{\kappa}_\infty(\m{X},\m{Y}),\label{gi2i}
\end{align}
where $\tilde{\kappa}_\infty(\m{X},\m{Y})$ is of order $ O(\varepsilon^{-1}\exp(-\varepsilon^{-1}))$. Then we obtain all desired expansions.
 \end{proof}
 We denote by  $K$, $K_\infty$, $\tilde{K}_\infty$ the integral operators corresponding to the Schwarz integral kernels $\kappa$, $\kappa_\infty: = r_1^\varepsilon+r_2^\varepsilon$, $\tilde{\kappa}_\infty $ respectively, where $r_1^\varepsilon$, $r_2^\varepsilon$ and $\tilde{\kappa}_\infty$ are defined in Lemma~\ref{LM1}. We further introduce an operator $P: V_1\rightarrow V_2,$ 
  $$ P\psi(\m{X})=\langle\psi,1\rangle1,$$ 
where $1\in V_2$ is the constant function equal to 1 on $D(0,1)$.
With these expansions of kernels in Lemma~\ref{LM1}, we can derive the following useful lemma for the integral operators in the equation \eqref{initialmatrix}:
\begin{lemma}\label{decomp}
\begin{enumerate}
\item The operator  $T^e+ T^i$ admits the following decomposition
$$T^e+ T^i = \beta(k,\varepsilon) P+K+K_\infty.$$
Moreover, $K_\infty$ is bounded from $V_1$ to $V_2$ with the operator norm $\Vert K_\infty\Vert \leq O(k^2 \varepsilon)$ uniformly for bounded $k$.
\item  The operator $\tilde{T}^i$ admits the following decomposition
$$T^e+ T^i = \tilde{\beta}(k,\varepsilon) P+\tilde{K}_\infty.$$
Moreover, $\tilde{K}_\infty$ is bounded  from $V_1$ to $V_2$ with the operator norm $\Vert \tilde{K}_\infty\Vert \leq O(\varepsilon^{-1} \exp(-\varepsilon^{-1}))$ uniformly for bounded $k$.
\item The operator $\tilde{K} = \varepsilon K$ is bounded from $V_1$ to $V_2$ with a bounded inverse. Moreover, $$\alpha := \langle\tilde{K}^{-1}1,1\rangle\neq 0.$$
\end{enumerate}
\end{lemma}
The first two results in Lemma~\ref{decomp} directly follow the definitions of $T^e$, $T^i$, $\tilde{T}^i$ and the expansions of their kernels in Lemma~\ref{LM1}. The proof of the third assertion will be presented in the next subsection.

\end{subsection}
\begin{subsection}{Invertibility of operator $\tilde{K}$}
In this subsection, we prove that the integral operator $\tilde{K}$ defined in Lemma~\ref{decomp} has a bounded inverse, and  $\alpha: =  \langle\tilde{K}^{-1}1,1\rangle\neq 0$.
We will carry out the asymptotic expansion of $\tilde{K}$ over $k\varepsilon$, denote the leading-order term of  $\tilde{K}$ as $\tilde{K}_0$, and analyze the invertibility of the integral operator $\tilde{K}_0$. Then we will obtain the conclusions for $\tilde{K}$.

One way to derive the leading-order term $\tilde{K}_0$ is carrying out the asymptotic expansion of the underlying partial differential equations. Recall that the kernel of $\tilde{K}$ is
\begin{equation*}
\varepsilon\kappa(k;\m{X},\m{Y}) =-\frac{1}{2\pi |\m{X}-\m{Y}|} +\varepsilon g_0^1(k;(\varepsilon\m{X},1),(\varepsilon\m{Y},1)),\end{equation*}
where the Green's function $g_0^1(k;\m{r},\m{r}')$ satisfies \eqref{waveguide}
with decaying boundary condition when $z\to -\infty$. 
Then we introduce an auxiliary function $G_0^1(k;\m{R},\m{R}')$ which satisfies
\begin{equation}
\begin{cases}
\Delta  G_0^1(k;\m{R},\m{R}')+\varepsilon^2 k^2 G_0^1(k;\m{R},\m{R}')= \delta(\m{R}-\m{R}'),\quad &\m{R},\m{R}'\in \Omega_1, \label{tg1}\\
\dfrac{\partial G_0^1 }{\partial\nu_{\m{R}'}} = 0,\quad &\m{R}'\in \ \partial\Omega_1,
\end{cases}
\end{equation}
and $\frac{\partial G_0^1}{\partial z}\to 0$ as $z\to -\infty$, where $\Omega_1$ denotes a semi-infinite waveguide 
$$\Omega_1 : = \{ \m{R}=(R,\theta, z):  0< R< 1,  -\pi< \theta\leq \pi, z< 1\}.$$
By the definitions of $g_0^1(k;\m{r},\m{r}')$ and the change of variable, there holds
$$G_0^1(k;(\m{X},1),(\m{Y},1))=\varepsilon g_0^1(k;(\varepsilon\m{X},1),(\varepsilon\m{Y},1)).$$     
On the other hand, one can observe from \eqref{tg1} that the function $G_0^1(k;\m{R},\m{R}')$ is actually the Green's function for the Helmholtz equation in the structure $\Omega_1$ with wavenumber equal to $k\varepsilon$ and the homogeneous Neumann boundary condition. Let us do asymptotic expansion on both sides of \eqref{tg1} as $k \varepsilon\ll 1$.
Using $\tilde{G}_0^1(\m{R},\m{R}')$ to denote the leading-order term of $G_0^1(k;\m{R},\m{R}')$ in the asymptotic expansion over $k\varepsilon$, then $\tilde{G}_0^1(\m{R},\m{R}')$ satisfies
\begin{equation}
\begin{cases}
\Delta  \tilde{G}_0^1(\m{R},\m{R}') = \delta(\m{R}-\m{R}'),\quad& \m{R},\m{R}'\in \Omega_1, \label{tg2}\\
\dfrac{\partial \tilde{G}_0^1 }{\partial\nu_{\m{R}'}} = 0,\quad& \m{R}'\in\ \partial\Omega_1,
\end{cases}
\end{equation}
 and $\frac{\partial \tilde{G}_0^1}{\partial z}\to 0$ as $z\to -\infty$. Thus we can conclude that the kernel of the integral operator  $\tilde{K}_0$ is $-\frac{1}{2\pi |\m{X}-\m{Y}|}+\tilde{G}_0^1((\m{X},1),(\m{Y},1))$.

Another way to obtain the leading-order term of  $\tilde{K}$ is deriving the asymptotic expansion of the kernel $\varepsilon\kappa : = -\frac{1}{2\pi|\m{X}-\m{Y}|}+\varepsilon r^i(k;\m{X},\m{Y})$ directly.  
For this purpose, we introduce the following notations of scaled functions,
\begin{align}
D_{mn} &:= \frac{1}{\int_0^1 \int_0^{2\pi}(J_n(q_{mn} R))^2
\cos^2(n\theta)d\theta RdR},\label{defD}\\
\Phi_{mne}(\m{X})&:=\phi_{mne}(\varepsilon \m{X}),\nonumber\\
\Phi_{mno}(\m{X})&:=\phi_{mno}(\varepsilon \m{X}),\nonumber
\end{align}
where $D_{mn} = \varepsilon^2 d_{mn}$ and $D_{mn}=O(1)$ by definition. 
Thus we can rewrite the formula of $\varepsilon r^i(k;\m{X},\m{Y})$ as 
\begin{equation*}
\varepsilon r^i(k;\m{X},\m{Y})=\sum_{(m, n)\neq (1,0)} \sum_{\xi \in\{o,e\}}\frac{-D_{mn}}{\sqrt{ (q_{mn})^2-(\varepsilon k)^2}} \Phi_{mn\xi}(\m{X})\Phi_{mn\xi}(\m{Y}),
\end{equation*}
and the asymptotic expansion over $k\varepsilon$ for the formula above takes the following form:
\begin{equation}
\varepsilon r^i(k;\m{X},\m{Y})=\sum_{(m, n)\neq (1,0)}\sum_{\xi \in\{o,e\}} \frac{-D_{mn}}{ q_{mn}} \Phi_{mn\xi}(\m{X})\Phi_{mn\xi}(\m{Y})+O(k^2\varepsilon^2).\label{formri}
\end{equation}
We observe that the summation in  \eqref{formri} is the leading-order term of $\varepsilon r^i(k;\m{X},\m{Y})$, which also coincides with the formula of the  Green's function $\tilde{G}_0^1(\m{R},\m{R}')$ defined in \eqref{tg2}.

Next, to investigate the invertibility of $\tilde{K}_0$, we introduce a  transmission problem related to the kernel of $\tilde{K}_0$, and prove the invertibility of $\tilde{K}_0$ by the well-posedness of this problem. The transmission problem in the structure $\Omega_1\cup \Omega^+$  is modeled by the following equations:
\begin{equation}
\begin{cases}
\Delta u=0 &\text{ in } \Omega_1\cup \Omega^+, \\
\dfrac{\partial u}{\partial \nu} = 0  &\text{ on } \partial (\Omega_1\cup \Omega^+),\\
\lbrack u \rbrack_{\tilde{\Gamma}}=f,&\\
\Big\lbrack \dfrac{\partial u}{\partial \nu} \Big\rbrack_{\tilde{\Gamma}}=0,&\\
\end{cases}\label{k0eq1}
\end{equation}
where $\tilde{\Gamma}: = \{ \m{R}=(R_0, \theta, z): 0\leq R_0\leq 1,-\pi<\theta\leq \pi, z=1 \}$, $\lbrack \cdot \rbrack_{\tilde{\Gamma}}$ is the conventional notation of jump that denotes the difference across the interface $\tilde{\Gamma}$ of $\Omega^+$ and $\Omega_1$, $u$ satisfies the radiating boundary condition in the upper half space, and $\frac{\partial u}{\partial z}\to 0$ as $z\to -\infty$. 
  it can be proved analogously to the transmission problem (P1) in  \cite{bonnetier2010asymptotic} that this system \eqref{k0eq1} is well-posed for $f\in H^{1/2}(\tilde{\Gamma})$, i.e., the unique solution to \eqref{k0eq1} exists and is bounded by $f$.
   On the other hand, we can use the same argument as in Subsection~\ref{subint} of deriving the boundary-integral equations \eqref{upper}--\eqref{lower} to deduce that $\tilde{K}_0$ maps the outer normal derivative of the solution for \eqref{k0eq1},  $\frac{\partial u}{\partial\nu}|_{\tilde{\Gamma}}\in H^{1/2}(\tilde{\Gamma})$, to the difference of the solution across the aperture,  $\lbrack u \rbrack_{\tilde{\Gamma}}=f$. More precisely, the invertibility of $\tilde{K}_0$ is equivalent to the well-posedness of the diffraction problem  \eqref{k0eq1}. Therefore we conclude that $\tilde{K}$ has a bounded inverse.

Finally, to prove  $\alpha =  \langle\tilde{K}^{-1}1,1\rangle\neq 0$, we first claim that there exists a constant $\alpha_0: =  \langle\tilde{K}_0^{-1}1,1\rangle\neq 0$, and prove the assertion by contradiction. If $\alpha_0=0$, taking  $f\equiv 1$ in the system \eqref{k0eq1} leads to 
\begin{equation}\label{intjump}
\int_{\tilde{\Gamma}} \dfrac{\partial u}{\partial \nu}\lbrack u \rbrack_{\tilde{\Gamma}} ds =\langle\tilde{K}_0^{-1}1,1\rangle=0.
\end{equation}
Using Green's theorem and \eqref{intjump}, one can derive that the integral of $(\nabla u)^2$ over a half ball $B_R^+ =\{ \m{R}=(R_0,\theta, z):|(R_0,\theta, z-1)|< R, -\pi<\theta\leq \pi, z>1\}$ coupled with a waveguide $D(0,1)\times (-R,1)$ satisfies
\begin{equation}
\int_{B_R^+ \cup (D(0,1)\times (-R,1))} (\nabla u)^2 ds =\int_{S_R^+} \nabla u\cdot \frac{\m{R}-(0,0,1)}{| \m{R}-(0,0,1)|} u ds-\int_{\Gamma_R}\frac{\partial u}{\partial z} u ds,  \label{inthalf}
\end{equation}
 where $S_R^+:=\{\m{R}=(R_0,\theta, z): |(R_0,\theta, z-1)|=R, -\pi<\theta\leq \pi, z\geq 1\}$ and $\Gamma_R: = \{\m{R}=(R_0,\theta, z): R_0\leq 1,-\pi<\theta\leq \pi,  z=-R\}$. When $R\to\infty$, with the Sommerfeld radiation condition in the infinity and $u\to 0$ as $z\to -\infty$, the right-hand side of \eqref{inthalf} converges to zero, which leads to $u\equiv 0$ in $\Omega_1\cup\Omega^+$, thus $\tilde{K}_0^{-1}1 =0$. This contradicts to the invertibility of $\tilde{K}_0$. Finally, as $\tilde{K}_0$ is the leading-order term of $\tilde{K}$, it follows $\alpha_0\neq 0$  that  $\alpha =  \langle\tilde{K}^{-1}1,1\rangle\neq 0$.

\end{subsection}
\end{section}
\begin{section}{Asymptotic expansions of resonances}\label{secresonance}
 In this section we investigate the resonant frequencies of the scattering problem \eqref{inieq}, and further derive their asymptotic expansions utilizing the properties of the integral operators in previous sections. As mentioned in the Introduction, the resonance is defined as a complex number $k$ with negative imaginary part such that there exists a nontrivial solution to the integral equations with wavenummber $k$ when the incidental wave is zero. Therefore, deriving the resonant frequencies is equivalent to finding the characteristic values of the operator on the left-hand side of \eqref{initialmatrix} with respect to the variable $k$ \cite{ammari2009layer}. In what follows, when there is no ambiguity, we omit the explicit dependence on $k$ and $\varepsilon$ to simplify the notations. In particular, we denote $\beta(k,\varepsilon)$ in \eqref{notbeta} by $\beta$ and $\tilde{\beta}(k,\varepsilon)$ in \eqref{nottbeta} by $\tilde{\beta}$ respectively.  We shall denote the operator on the left-hand side of \eqref{initialmatrix} by $\mathbb{T}$, which admits the following decomposition by Lemma~\ref{decomp},
\begin{equation}\label{notofT}
\mathbb{T}:=
\begin{bmatrix}
T^e+ T^i& \tilde{T}^i\\
\tilde{T}^i& T^e+ T^i
\end{bmatrix}=\begin{bmatrix}
\beta P& \tilde{\beta} P\\
\tilde{\beta} P& \beta P
\end{bmatrix}+ K\mathbb{I}+
\begin{bmatrix}
K^\infty& \tilde{K}^\infty\\
\tilde{K}^\infty& K^\infty
\end{bmatrix}.
\end{equation}
 Let us further introduce the notations
\begin{equation*}
\mathbb{P}:=
\begin{bmatrix}
\beta P& \tilde{\beta} P\\
\tilde{\beta} P& \beta P
\end{bmatrix}, \ \mathbb{L}:=
 K\mathbb{I}+
\begin{bmatrix}
K^\infty& \tilde{K}^\infty\\
\tilde{K}^\infty& K^\infty
\end{bmatrix},
\end{equation*}
and denote by $\sigma(\mathbb{T})$ the set of complex characteristic frequencies $k$ of the operator $\mathbb{T}$. 
We can decompose the space $V_1\times V_1$ into two invariant spaces of $\mathbb{T}$, i.e.,  $V_1\times V_1= V_{even}\oplus V_{odd}$, where
\begin{align*}
V_{even} &= \{ [\Psi,\Psi]^T, \Psi\in V_1\},\\
V_{odd} &=  \{ [\Psi,-\Psi]^T, \Psi\in V_1\},
\end{align*}
  and consequently  $\sigma(\mathbb{T})$ could be divided as 
 $$\sigma(\mathbb{T})=\sigma(\mathbb{T}|_{V_{even}})\cup \sigma(\mathbb{T}|_{V_{odd}}) = \sigma(\mathbb{T}^+)\cup \sigma(\mathbb{T}^-),$$
 where
 \begin{align*}
 \mathbb{T}^+& := T^e+ T^i + \tilde{T}^i,\\
 \mathbb{T}^- &:= T^e+ T^i - \tilde{T}^i.
 \end{align*}
 Following the asymptotic expansion in Lemma~\ref{decomp}, we have
 \begin{align*}
 \mathbb{T}^+& = \left(\beta_1+\beta_2+\tilde{\beta}\right)P + K+K_\infty+\tilde{K}_\infty: = \mathbb{P}^+ +\mathbb{L}^+,\\
 \mathbb{T}^-& = \left(\beta_1+\beta_2-\tilde{\beta}\right)P + K+K_\infty-\tilde{K}_\infty: = \mathbb{P}^- +\mathbb{L}^-,
 \end{align*}
 where $\mathbb{P}^+ =  (\beta_1+\beta_2+\tilde{\beta})P$, $\mathbb{L}^+ = K+K_\infty+\tilde{K}_\infty$, $\mathbb{P}^- =  (\beta_1+\beta_2-\tilde{\beta})P$, and $\mathbb{L}^- = K+K_\infty-\tilde{K}_\infty$.  From the asymptotic expansions in Lemma~\ref{decomp}, we can deduce that  $\mathbb{L}^+$, $\mathbb{L}^-$ are invertible for sufficiently small $\varepsilon$. 
 
 First, let us consider the characteristic frequencies in $\sigma(\mathbb{T}^+)$. Assuming that there exists a function $\varphi\in V_1$ such that $\varepsilon(\mathbb{P}^+ +\mathbb{L}^+)\varphi=0$,
 then
 \begin{equation*}
 (\varepsilon\mathbb{L}^+)^{-1}(\varepsilon\mathbb{P}^+) \varphi +\varphi = 0,
 \end{equation*}
i.e.,
  \begin{equation*}
\varepsilon(\beta_1+\beta_2+\tilde{\beta}) (\varepsilon\mathbb{L}^+)^{-1} P\varphi +\varphi = 0.
 \end{equation*}
It is straightforward to check that the eigenvalue of $\varepsilon (\beta+\tilde{\beta}) (\varepsilon\mathbb{L}^+)^{-1} P+I$ is
 \begin{equation}\label{eqlambda1}
 \lambda_1(k,\varepsilon):= 1+\varepsilon(\beta+\tilde{\beta}) \langle(\varepsilon\mathbb{L}^+)^{-1}1,1\rangle.
 \end{equation}
 Similarly,  we denote the eigenvalue corresponding to $\mathbb{T}^-$ as
  \begin{equation}\label{eqlambda2}
  \lambda_2(k,\varepsilon):= 1+\varepsilon(\beta-\tilde{\beta}) \langle(\varepsilon\mathbb{L}^-)^{-1}1,1\rangle. 
  \end{equation}
  Therefore, the characteristic frequencies of $\mathbb{T}$ are the roots of these two analytic functions $\lambda_1(k,\varepsilon)$ and  $\lambda_2(k,\varepsilon)$, which can be concluded as the following lemma.
   \begin{lemma}\label{lemlambda}
 The resonances of the problem  \eqref{initialmatrix} are the roots of the analytic functions
 $$\lambda_1(k,\varepsilon)=0$$ and $$\lambda_2(k,\varepsilon)=0.$$
    \end{lemma}
     Finally, we are ready to present the main theoretical results of this section, the asymptotic expansions for the resonances of the system \eqref{initialmatrix}. We shall restrict the discussion to the right half of the complex $k$ plane, and the resonances on left half can be derived analogously.
  \begin{theorem*}\label{thmreso}
   There exist two sets of resonances $\{k_{n,1}\}$ and $\{k_{n,2}\}$ for the scattering problem \eqref{initialmatrix}, and the following asymptotic expansions hold:
  \begin{align*}
  k_{n,1} &= (2n-1)\pi +\frac{2(2n-1)\pi\varepsilon}{\alpha}-(2n-1)^2\pi i\varepsilon^2 +O(\varepsilon^3),\\
  k_{n,2}&= 2n\pi+ 4n\pi \frac{\varepsilon}{\alpha}-4n^2\pi i \varepsilon^2 +O(\varepsilon^3),
  \end{align*}
  where  $\alpha=\la\tilde{K}^{-1}1,1\ra$, $n\geq 1$, and $n\varepsilon\ll 1$.
  \end{theorem*}
  \begin{proof}
  Following Lemma~\ref{lemlambda}, we derive the characteristic frequencies of $\mathbb{T}^+$ by solving
  \begin{equation*}
  \lambda_1(k,\varepsilon) = 1+\varepsilon(\beta+\tilde{\beta}) \langle(\varepsilon\mathbb{L}^+)^{-1}1,1\rangle =0.
  \end{equation*}
 From Lemma~\ref{decomp},  $K_\infty$ and $\tilde{K}_\infty$ are bounded in the operator norm $\Vert K_\infty\Vert \leq O(k^2 \varepsilon)$, $\Vert \tilde{K}_\infty\Vert \leq O(\varepsilon^{-1}\exp(-\varepsilon^{-1}))$.  
 Then it follows $\varepsilon\mathbb{L}^+ = \tilde{K}+\varepsilon K_\infty+\varepsilon \tilde{K}_\infty$  that
  \begin{align}
   \langle(\varepsilon\mathbb{L}^+)^{-1}1,1\rangle&=\langle(\tilde{K}+\varepsilon K_\infty+\varepsilon\tilde{K}_\infty)^{-1}1,1\rangle\nonumber\\
   &=\alpha\left(1+\sum_{n=1}^\infty \langle-\tilde{K}^{-1}(\varepsilon K_\infty+\varepsilon\tilde{K}_\infty)^n(1),1\rangle\right)\nonumber\\
   &:=\alpha+r(k,\varepsilon),\label{Linverse}
  \end{align}    
   where $r(k,\varepsilon)$ is analytic in $k$  and $r(k,\varepsilon)\sim O(k^2 \varepsilon^2)$. For simplicity, we use the notations
  \begin{align}
  c(k) &:= \frac{\text{cot}\, k}{k}+\frac{1}{k\sin k},\label{notc}\\
  \gamma(k) &:=-\frac{ik}{2\pi}.\label{notgamma}
  \end{align}
   Introducing $p(k,\varepsilon) := \varepsilon \lambda_1(k,\varepsilon)$,  we rewrite the equation $\lambda_1(k,\varepsilon)=0$ with the notations above as 
  \begin{equation}
  p(k,\varepsilon) = \varepsilon+ \left(c(k)+\varepsilon^2 \gamma(k)\right)\left(\alpha+r(k,\varepsilon)\right)=0.\label{formp}
  \end{equation}
   Since the resonances under our consideration are not in the high frequency regime, we search for roots of $p(k,\varepsilon)$ in a bounded domain $ D_{M} = \{z: |z|\leq M\}$ for some fixed constant $M>0$.
  Note that $p(k,\varepsilon)$ blows up as $k\to 2j\pi$, for all $j\in \mathbb{Z}$, where $\mathbb{Z}$ denotes the set of all integers. As the result, there exists $\delta_0$ such that all roots of $p(k,\varepsilon)$ in $D_{M}$ actually lie inside a smaller domain 
  \begin{equation*}
  D_{\delta_0, M} = \{z: |z-2j\pi|\geq \delta_0,\ \forall j\in\mathbb{Z}\}\cap D_{M},
  \end{equation*}
  and $p(k,\varepsilon)$ is analytic for $k$ in this subdomain $D_{\delta_0,M}$.

  We can observe from the formula \eqref{notc} that the roots of $c(k)=0$ on the right half of the complex $k$ plane are $k_{n,1,0} = (2n-1)\pi$ with $n\geq 1$ being an integer, and each root is simple. Following \eqref{formp}, we can apply Rouche's theorem to deduce that the roots of $\lambda_1(k,\varepsilon)$, denoted as $k_{n,1}$, are also simple and very close to $k_{n,1,0}$ when $\varepsilon\ll 1$. We shall derive the asymptotic expansions for these roots below.

  Let us first introduce a function
  \begin{equation}
  p_1(k,\varepsilon) = \varepsilon+ \left(c(k)+\varepsilon^2 \gamma(k)\right)\alpha,\label{p1rep}
  \end{equation}
  and present the Taylor's expansion of $ p_1(k,\varepsilon)$ at $k_{n,1,0}$,
  \begin{align*}
  p_1(k,\varepsilon)=&\varepsilon+ \big(c'(k_{n,1,0})(k-k_{n,1,0})+c''(k_{n,1,0})(k-k_{n,1,0})^2+O(k-k_{n,1,0})^3\\
  &+\varepsilon^2 \gamma(k_{n,1,0})+\varepsilon^2  \gamma'(k_{n,1,0})(k-k_{n,1,0})\big)\alpha.
  \end{align*}
  By the formula of $c(k)$ we have $c'(k_{n,1,0}) =-\frac{1}{2k_{n,1,0}} =-\frac{1}{2(2n-1)\pi}$. We can further calculate $c''(k)$,
  \begin{equation*}
  c''(k)=\frac{2k\sin^3 k+((2k^2+2)\cos k+k^2+2)\sin^2 k+(2k\cos^2 k +2k \cos k)\sin k+2k^2 \cos^3 k +2k^2\cos^2 k }{k^3\sin^3 k },
  \end{equation*}
  and it follows that 
$c''(k_{n,1,0})=0$.
 Thus we conclude that the roots of $p_1(k,\varepsilon)$ in $D_{\delta_0,M}$ are also simple and close to $k_{n,1,0}$.  Moreover, these roots are analytic with respect to the variable $\varepsilon$, and we denote the roots of $p_1(k,\varepsilon)$ which are close to $k_{n,1,0}$ as $k_{n,1,1}$. Expanding these roots in terms of powers of $\varepsilon$, we obtain from \eqref{p1rep} that
  \begin{equation*}
  k_{n,1,1} = k_{n,1,0}+2 k_{n,1,0} \frac{\varepsilon}{\alpha}-k_{n,1,0}^2\frac{ i\varepsilon^2}{\pi}+O(\varepsilon^3).
  \end{equation*}   
  Then we claim that $k_{n,1,1}$ gives the leading-order terms of $k_{n,1}$, more precisely, 
  \begin{equation*}
 k_{n,1,1}-k_{n,1} = O(\varepsilon^3).
  \end{equation*}
  We prove this claim by Rouche's theorem. As
  \begin{equation*}
  p(k,\varepsilon) - p_1(k,\varepsilon) = O(c(k)+\varepsilon^2 \gamma(k))r(k,\varepsilon)
  \end{equation*}
  and 
  $$p_1(k,\varepsilon) = c(k)\alpha+O(\varepsilon),$$
  one can find a positive constant $C_n$ such that 
  \begin{equation}
  |p(k,\varepsilon) - p_1(k,\varepsilon)|\leq |p_1(k,\varepsilon)|
  \end{equation}
  for all $|k-k_{n,1,1}|= C_n \varepsilon^3$. Thus we obtain that $p(k,\varepsilon)$ has a simple root in $\{k:|k-k_{n,1,1}|\leq C_n \varepsilon^3\}$, which implies $ k_{n,1,1}-k_{n,1} = O(\varepsilon^3)$.
  
  Similarly, considering
  $$\lambda_2(k,\varepsilon) = 1+\varepsilon(\beta-\tilde{\beta}) \langle(\varepsilon\mathbb{L}^-)^{-1}1,1\rangle, $$ 
  we introduce a function $q(k,\varepsilon) := \varepsilon \lambda_2(k,\varepsilon)$ and by definition
  \begin{equation*}
  q(k,\varepsilon) = \varepsilon+ \left(\frac{\text{cot}\, k}{k}-\frac{1}{k\sin k}- \frac{ik}{2\pi}\varepsilon^2\right)(\alpha+s(k,\varepsilon)),
  \end{equation*}
  where $s(k,\varepsilon)$ is analytic in $k$  and $s(k,\varepsilon)\sim O(k^2 \varepsilon^2)$.
  One can deduce the asymptotic expansions for the roots to $q(k,\varepsilon) =0$ as
  $$ k_{n,2}= 2n\pi+ 4n\pi \frac{\varepsilon}{\alpha}-4 n^2\pi i \varepsilon^2 +O(\varepsilon^3),$$
  which completes the proof.
   \end{proof}
   
Through out this work, we will call the frequencies $\{k_{n,1}\}$ odd resonant frequencies and $\{k_{n,2}\}$  even resonant frequencies. Moreover, for the case when thickness of the plate $\ell \neq 1$ and $\varepsilon\ll \ell$, the expansions in Theorem~\ref{thmreso} for $\ell=1$ can be generated by a scaling argument and concluded as the following proposition.

\begin{proposition*}\label{thmresothick}
   There exist two sets of resonances $\{k_{n,1}\}$ and $\{k_{n,2}\}$ for the scattering problem \eqref{initialmatrix}, and the following asymptotic expansions hold:
  \begin{align*}
  k_{n,1} &= \frac{1}{\ell}\left((2n-1)\pi +2(2n-1)\pi\frac{\varepsilon}{\alpha}-(2n-1)^2\pi i \varepsilon^2 +O(\varepsilon^3)\right),\\
  k_{n,2}&=\frac{1}{\ell}\left( 2n\pi+ 4n\pi \frac{\varepsilon}{\alpha}-4n^2\pi  i \varepsilon^2 +O(\varepsilon^3)\right),
  \end{align*}
  where  $\alpha=\la\tilde{K}^{-1}1,1\ra$, $n\geq 1$, and $n\varepsilon\ll 1$.
  \end{proposition*}

\end{section}
\begin{section}{Quantitative analysis of the field enhancement at resonant frequencies}\label{quantreso}
 We will carry out the quantitive analysis of the field enhancement when $k$ takes the value of a resonant frequency in this section. The asymptotic behavior of transmitted fields in the far field and in the circular hole will be investigated  with both the enhancement order and shapes of the resonant modes characterized. 
\begin{subsection}{Preliminaries}
We start with some preliminary calculations of the system \eqref{initialmatrix}. Recall the definitions and asymptotic expansions of  $p(k,\varepsilon)$ and $q(k,\varepsilon)$ in Section~\ref{secresonance},  
  \begin{align}
  p(k,\varepsilon) = \varepsilon \lambda_1(k,\varepsilon) = \varepsilon+ \left(\frac{\text{cot} \,k}{k}+\frac{1}{k\sin k}-\frac{ik}{2\pi}\varepsilon^2 \right)(\alpha+r(k,\varepsilon)),\label{expofp}\\
  q(k,\varepsilon) = \varepsilon \lambda_2(k,\varepsilon) = \varepsilon+ \left(\frac{\text{cot}\, k}{k}-\frac{1}{k\sin k}- \frac{ik}{2\pi}\varepsilon^2\right)(\alpha+s(k,\varepsilon)),\label{expofq}
  \end{align}
where $r(k,\varepsilon)\sim O(k^2\varepsilon^2)$, $s(k,\varepsilon)\sim O(k^2\varepsilon^2)$. We shall study their asymptotic behavior at the odd and even frequencies in the following lemma.
\begin{lemma} \label{enhancepq}
If $n\varepsilon\ll 1$, then at the odd and even resonant frequencies $k = \operatorname{Re} k_{n,1}$ and 
$k = \operatorname{Re} k_{n,2}$, there hold
\begin{align*}
p(k,\varepsilon) &=-\frac{\alpha(2n-1)\pi i}{2}\varepsilon^2+O(\varepsilon^3),\\
q(k,\varepsilon) &=-\alpha n\pi i \varepsilon^2+O(\varepsilon^3),\end{align*}
 respectively.
\end{lemma}
\begin{proof}
First we consider $p(k,\varepsilon)$. For any $k$ satisfies $|k-\operatorname{Re} k_{n,1}|\leq O(\varepsilon)$, following notations in \eqref{formp} and \eqref{p1rep}, we can do expansion of $p(k,\varepsilon)$ at $k_{n,1}$:
\begin{align*}
p(k,\varepsilon) &=p_1(k)+c(k)O(\varepsilon^2) \\
&=p_1(k_{n,1})+p_1'(k_{n,1})(k-k_{n,1})+p_1''(k_{n,1})(k-k_{n,1})^2+c(k_{n,1})O(\varepsilon^2)+O(\varepsilon^3)\\
&=\alpha\left(c'(k_{n,1})+\varepsilon^2 \gamma'(k_{n,1})\right)(k-k_{n,1})+\alpha c''(k_{n,1})(k-k_{n,1})^2+O(\varepsilon^3)\\
&=\alpha c'(k_{n,1,0})(k-k_{n,1})+O(\varepsilon^3)\\
&=\alpha \left(-\frac{1}{2k_{n,1,0}}\right)(k-\operatorname{Re} k_{n,1} -i \operatorname{Im} k_{n,1})+O(\varepsilon^3),
\end{align*}
where we have used the expansions of $k_{n,1}$ in Theorem~\ref{thmreso} and the  estimates of $c(k_{n,1})$.
If  $k = \operatorname{Re} k_{n,1}$, then it follows from Theorem ~\ref{thmreso} that 
$$k-\operatorname{Re} k_{n,1} -i \operatorname{Im} k_{n,1} = k_{n,1,0}^2\frac{ i\varepsilon^2}{\pi}+O(\varepsilon^3).
 $$
which leads to 
$$p(k,\varepsilon)=-\alpha(2n-1)\pi\frac{ i\varepsilon^2}{2}+O(\varepsilon^3).$$
Similarly for the function $q(k,\varepsilon)$, we obtain that when $k = \operatorname{Re} k_{n,2}$, 
$$q(k,\varepsilon)=-\alpha n\pi i \varepsilon^2+O(\varepsilon^3).$$

\end{proof}

Again, we omit the explicit dependence on $k$ and $\varepsilon$ to simplify the notations and denote $p(k,\varepsilon)$ by $p$ and $q(k,\varepsilon)$ by $q$ respectively, when there is no risk of causing confusion.
With Lemma~\ref{enhancepq}, we can deduce the asymptotic expansion of the solution $\m{\Psi}=[\Psi_1,\Psi_2]^T$ to the system of integral equations \eqref{initialmatrix}, and carry out the quantitative analysis of the field enhancement as the following lemma.
\begin{lemma}\label{enhancementpsi}
 The following asymptotic expansion holds for the solution $\m{\Psi}=[\Psi_1,\Psi_2]^T$ to \eqref{initialmatrix}:
\begin{equation}\label{exppsi}
 \m{\Psi} = \begin{bmatrix}
 \alpha\tilde{K}^{-1}1 \cdot\left(d_1\cdot O(k)+\frac{1}{p}+\frac{1}{q}\right)+(\frac{1}{p}+\frac{1}{q})\left( d_1 \cdot O(k\varepsilon)+ O(k^2\varepsilon^2)\right)+ d_1\cdot O(k^3\varepsilon^2) \\
 \alpha \tilde{K}^{-1}1 \cdot(\frac{1}{p}-\frac{1}{q})+(\frac{1}{p}-\frac{1}{q})\left( d_1 \cdot O(k\varepsilon)+ O(k^2\varepsilon^2)\right)+ d_1 \cdot O(k^3\varepsilon^2)
 \end{bmatrix},
 \end{equation}
 where  $\m{d} = (d_1,0, d_3)$ denotes the incidence direction. Moreover, 
\begin{align}
\langle \Psi_1, 1\rangle =(\alpha + d_1 \cdot O(k\varepsilon) + O(k^2\varepsilon^2)) \left(\frac{1}{p}+\frac{1}{q}\right),\label{gamma1int}  \\
\langle \Psi_2, 1\rangle =(\alpha + d_1\cdot  O(k\varepsilon) + O(k^2\varepsilon^2)) \left(\frac{1}{p}-\frac{1}{q}\right).\label{gamma2int}
\end{align}
\end{lemma}
\begin{proof}
As the incident field $u^i(\m{r}) = e^{ik (d_1 r\cos\theta -d_3(z-1))}$ and the reflected field  $u^r(\m{r}) = e^{ik (d_1 r\cos\theta+d_3(z-1))}$, we present the asymptotic expansion of $f$ in \eqref{initialmatrix} as
\begin{equation*}
f(\m{X})= (u^i+u^r)(\varepsilon\m{X},1)= 2+ d_1 \cdot O(k\varepsilon).
\end{equation*}
 With the notations in \eqref{notofT}, the system \eqref{initialmatrix} could be denoted as 
\begin{equation*}
\mathbb{T}\m{\Psi} = \m{f},
\end{equation*}
where $\m{f} = [\varepsilon^{-2}f,0]^T$.
Applying the decomposition of the space $V_1\times V_1$ in Section \ref{secresonance}, there holds $\m{f} = \m{f}_{odd}+\m{f}_{even}$, where $\m{f}_{odd} = [\bar{f}, -\bar{f}]^T\in V_{odd}$, $\m{f}_{even} = [\bar{f}, \bar{f}]^T\in V_{even}$, and
\begin{equation}
\bar{f} =\frac{1}{2\varepsilon^2}( 2+ d_1 \cdot O(k\varepsilon)).\label{repf}
\end{equation}
 Next we solve the subsystems $\mathbb{T} \Psi_{odd} = \m{f}_{odd}$,  $\mathbb{T} \Psi_{even} = \m{f}_{even}$ separately, which are equivalent to two smaller systems
\begin{align}
\mathbb{T}^+\Psi^+ =  \bar{f}, \label{T+eq}\\
\mathbb{T}^-\Psi^- = \bar{f}.\label{T-eq}
\end{align}
  Using the decomposition $\mathbb{T}^+=\mathbb{P}^++\mathbb{L}^+ $ and applying $(\mathbb{L}^+)^{-1}$ on both sides of \eqref{T+eq} lead to \begin{equation}
(\mathbb{L}^+)^{-1}\mathbb{P}^+ \Psi^++\Psi^+ = (\mathbb{L}^+)^{-1} \bar{f}.\label{psieq}
\end{equation} 
Recall that in Section~\ref{secresonance} we denote the eigenvalue of $(\mathbb{L}^+)^{-1}\mathbb{P}^++I$ as $\lambda_{1}(k,\varepsilon)$, which is  associated with the eigenvector $v = 1$. Now from \eqref{psieq} we can derive 
\begin{equation*}
\langle \Psi^+,1\rangle(\beta+\tilde{\beta})\langle(\mathbb{L}^+)^{-1}1,1\rangle+\langle \Psi^+,1 \rangle=\langle(\mathbb{L}^+)^{-1} \bar{f},1\rangle.
\end{equation*}
It follows that the solution $\Psi^+$ satisfies:
\begin{equation}
\langle \Psi^+, 1 \rangle =\frac{1}{\lambda_1(k,\varepsilon)} \langle(\mathbb{L}^+)^{-1} \bar{f},1 \rangle .\label{Psi1int} 
\end{equation}

From \eqref{psieq} one can also derive
\begin{align}
\Psi^+ &= (\mathbb{L}^+)^{-1} \bar{f}-(\mathbb{L}^+)^{-1}\mathbb{P}^+\Psi^+ \nonumber \\
&=(\mathbb{L}^+)^{-1} \bar{f}-\frac{\beta_1+\beta_2+ \tilde{\beta}}{\lambda_1(k,\varepsilon)}(\mathbb{L}^+)^{-1}\langle(\mathbb{L}^+)^{-1} \bar{f},1 \rangle 1\nonumber\\
&=(\varepsilon\mathbb{L}^+)^{-1} (\varepsilon\bar{f})+\frac{\left(1-\lambda_1(k,\varepsilon)\right)}{\lambda_1(k,\varepsilon) \langle(\mathbb{L}^+)^{-1} 1, 1\rangle}(\mathbb{L}^+)^{-1}\langle(\mathbb{L}^+)^{-1} \bar{f},1 \rangle, \label{Psirep}
\end{align}
where we have used the definition of $\lambda_1(k,\varepsilon)$.
Following the expansion \eqref{repf} and the definition of $\mathbb{L}^+$ , one obtain
 \begin{align}
(\mathbb{L}^+)^{-1}\bar{f} &=(\varepsilon\mathbb{L}^+)^{-1}\left(\frac{1}{2\varepsilon}( 2+ d_1 \cdot O(k\varepsilon))\right) \nonumber\\
&=\frac{1}{2\varepsilon}( 2+ d_1 \cdot O(k\varepsilon))(\tilde{K}^{-1}1 +O(k^2\varepsilon^2)).\label{expanL}
 \end{align}
Together with the expansions \eqref{Linverse} and \eqref{expanL}, from \eqref{Psirep} one can derive
\begin{equation*}
 \begin{aligned}
 2\Psi^+&=  2(\varepsilon\mathbb{L}^+)^{-1} (\varepsilon\bar{f})+\frac{2\left(1-\lambda_1(k,\varepsilon)\right)/ \langle(\varepsilon\mathbb{L}^+)^{-1}1, 1\rangle}{\varepsilon\lambda_1(k,\varepsilon)}(\mathbb{L}^+)^{-1}\langle(\mathbb{L}^+)^{-1}\bar{f},1 \rangle 1 \nonumber\\
 &= \frac{1}{\varepsilon}( 2+ d_1\cdot O(k\varepsilon))(\tilde{K}^{-1}1 +O(k^2\varepsilon^2))\nonumber\\
 &\quad  +\frac{\left(1-\lambda_1(k,\varepsilon)\right)\varepsilon}{(\alpha +r(k,\varepsilon))p(k,\varepsilon)}(\varepsilon\mathbb{L}^+)^{-1}\left(\frac{1}{\varepsilon}(2+ d_1 \cdot O(k\varepsilon))(\alpha +O(k^2\varepsilon^2))\right)\nonumber\\
 &=\frac{1}{\varepsilon}( 2+ d_1 \cdot O(k\varepsilon))(\tilde{K}^{-1}1  +O(k^2\varepsilon^2))\nonumber\\
 &\quad+\frac{(1-\lambda_1(k,\varepsilon))}{(\alpha+r(k,\varepsilon))p(k,\varepsilon)}( 2+ d_1 \cdot O(k\varepsilon))(\alpha+O(k^2\varepsilon^2)) \tilde{K}^{-1}1\nonumber\\
 &=\tilde{K}^{-1}1 \cdot\left(d_1\cdot O(k)+\frac{2\alpha}{p(k,\varepsilon)} \right)+\frac{\alpha}{p(k,\varepsilon)}( d_1\cdot O(k\varepsilon)+ O(k^2\varepsilon^2))+ d_1\cdot O(k^3\varepsilon^2).
  \end{aligned}
  \end{equation*}
  
 Similarly we deduce the expansion of the solution $\Psi^-$ below
  \begin{equation*}
 2\Psi^- =\tilde{K}^{-1}1 \cdot\left(d_1\cdot  O(k)+\frac{2\alpha}{q(k,\varepsilon)} \right)+\frac{\alpha}{q(k,\varepsilon)}( d_1 \cdot O(k\varepsilon)+ O(k^2\varepsilon^2))+ d_1  \cdot O(k^3\varepsilon^2),
 \end{equation*}
 which, together with the asymptotic expansion of $\Psi^+$, completes the proof of assertion \eqref{exppsi}.
  
From \eqref{Psi1int} and the asymptotic expansion \eqref{expanL} of $(\mathbb{L}^+)^{-1}\bar{f}$, one can derive the desired expansion for $\langle \Psi_1, 1\rangle$. The calculations for $\langle \Psi_2, 1\rangle$ are the same.
\end{proof}
Then the following proposition for $\m{\Psi}$ directly follows Lemma~\ref{enhancepq} and Lemma~\ref{enhancementpsi}.
 \begin{proposition*}\label{proponso}
At resonant frequencies, $\m{\Psi}\sim O(k^{-1}\varepsilon^{-2})$ in $V_1\times V_1$ , and $\langle \Psi_i, 1\rangle\sim O(k^{-1}\varepsilon^{-2})$, $i=1,2$.
\end{proposition*}
\end{subsection}

\begin{subsection}{Enhancement in the far field}
Let us investigate the scattered pressure field $u_\varepsilon^s$ in the upper domain and outside of the unit ball, i.e., $\Omega^+\backslash B_1^+$ at resonant frequencies, where $B_1^+ = \{\m{r}: |\m{r}-(0,0,1)|< 1\}$. From the representation formula we have for $\m{r}\in \Omega^+\backslash B_1^+$,
\begin{align*}
u_\varepsilon^s(\m{r}) &= \int_{\Gamma_1}g^e(k;\m{r},\m{r}')\frac{\partial u_\varepsilon^s}{\partial\nu}(\m{r}')ds(\m{r}')\\
&= - \varepsilon^2 \int_{D(0,1)}g^e(k;\m{r},(\varepsilon\m{X},1))\Psi_1 (\m{X})ds( \m{X}).
\end{align*}
 As $\varepsilon\ll 1$, we have the asymptotic expansion from the definition of $g^e(k;\m{r},\m{r}')$ that   
 $$g^e(k;\m{r},(\varepsilon\m{X},1)) = g^e(k;\m{r},(0,0,1))(1 +O(\varepsilon)).$$
 Together with  the approximation in Lemma~\ref{enhancementpsi}, we obtain for $\m{r}\in \Omega^+\backslash B_1^+$,
\begin{align}
u_\varepsilon^s(\m{r}) &=  -\varepsilon^2 g^e(k;\m{r},(0,0,1))(1 +O(\varepsilon))\int_{D(0,1)}\Psi_1 (\m{R}')ds( \m{R}')\nonumber \\
&=-\varepsilon^2 g^e(k;\m{r},(0,0,1))(\alpha + d_1 \cdot  O(k\varepsilon) + O(k^2\varepsilon^2)) \left(\frac{1}{p}+\frac{1}{q}\right)(1+O(\varepsilon))\nonumber \\
&=-\varepsilon^2 \alpha g^e(k;\m{r},(0,0,1)) \left(\frac{1}{p}+\frac{1}{q}\right)+ d_1 \cdot g^e(k;\m{r},(0,0,1)) \left(\frac{1}{p}+\frac{1}{q}\right) O(k\varepsilon^3).\label{farexpu}
\end{align}
 We readily observe from the formula \eqref{farexpu} that the scattered wave  in the far field behaves as the radiation field of a monopole located at $(0,0,1)$, and the amplitude of this field is  $\varepsilon^2 |\alpha||\frac{1}{p}+\frac{1}{q}|$.  When $k$ takes the value of a resonant frequency, the wave is enhanced due to the term $|\frac{1}{p}+\frac{1}{q}|$. As proved in Lemma~\ref{enhancepq}, at the odd and even resonant frequencies, the values of $\frac{1}{p}$ and $\frac{1}{q}$ are of order $O(k^{-1}\varepsilon^{-2})$ respectively. It directly follows that when $k=\operatorname{Re} k_{n,1}$ or $k=\operatorname{Re} k_{n,2}$, 
\begin{equation*}
u_\varepsilon^s(\m{r}) = \frac{2 i}{k} g^e(k;\m{r},(0,0,1))+O(k\varepsilon),
\end{equation*}
with the enhancement of order $O(\varepsilon^{-2})$ compared to the case when $k$ is not a resonant frequency, which demonstrates that  the field enhancement can be attributed to  the  Fabry-Perot type resonances.

Similarly, for $\m{r}\in\Omega^-\backslash B_1^-$, where $B_1^- = \{\m{r}=(r,\theta, z): |\m{r}-(0,0,0)|\leq 1\}$, one can deduce the formula for $u_\varepsilon^s(\m{r})$,
\begin{equation*}
u_\varepsilon^s(\m{r}) =-\varepsilon^2 \alpha g^e(k;\m{r},(0,0,0)) \left(\frac{1}{p}-\frac{1}{q}\right)+  d_1 \cdot g^e(k;\m{r},(0,0,0))   \left(\frac{1}{p}-\frac{1}{q}\right)O(k\varepsilon^3),
\end{equation*}
which illustrates that the scattered field is equivalent to the radiating field of a monopole located at $(0,0,0)$. Moreover, one can observe the field enhancement of order $O(\varepsilon^{-2})$ at the resonant frequencies due to the terms $\frac{1}{p}$ and $\frac{1}{q}$ in the formula.
\end{subsection}

\begin{subsection}{Enhancement in the hole}
%
Next we investigate the scattered field inside of the hole $\Omega^\varepsilon$, which can be represented as the summation of the eigenfunctions satisfying the homogeneous Neumann boundary conditions, that is, 
\begin{align}
u_\varepsilon^s(\m{r})&=  a_{1,0} \cos(kz)+b_{1,0}\cos(k(1-z))\nonumber \\
&\quad+ \sum_{(m,n)\neq(1,0)}\sum_{\xi \in\{o,e\}} a_{mn\xi}\phi_{mn\xi}(\m{r})\exp(-\gamma_{mn}z)\nonumber\\
&\quad +\sum_{(m,n)\neq(1,0)} \sum_{\xi \in\{o,e\}}b_{mn\xi} \phi_{mn\xi}(\m{r})\exp(-\gamma_{mn}(1-z)), \label{expancy}
\end{align}
where $\gamma_{mn}^2 = \alpha_{mn}^2 -k^2$, and $\gamma_{mn}$ has positive real part. When $k\varepsilon\ll 1$, the coefficients $\{a_{mno},a_{mne}, b_{mno}, b_{mne}\}$ have expansions as shown in the following lemma:
\begin{lemma}\label{a0b0}
The coefficients in the expansion \eqref{expancy} for $u_\varepsilon^s(\m{r})$  satisfy
\begin{align*}
a_{1,0} &= \frac{1}{k\sin k}(\alpha +d_1\cdot O(k\varepsilon) + O(k^2\varepsilon^2)) \left(\frac{1}{p}+\frac{1}{q}\right), \\
b_{1,0} &=\frac{1}{k\sin k}(\alpha +d_1\cdot O(k\varepsilon) + O(k^2\varepsilon^2)) \left(\frac{1}{p}-\frac{1}{q}\right),
\end{align*}
and
\begin{equation}
\sqrt{q_{mn}}a_{mn\xi}\varepsilon\leq C,\quad \sqrt{q_{mn}}b_{mn\xi}\varepsilon\leq C \label{estqmn}
\end{equation}
for $m\geq 1$, $n\geq 0$, $\xi \in\{o,e\}$ and some constant $C$.
\end{lemma}
\begin{proof}
  Representing $u_\varepsilon^s(\m{r})$ with the expansion \eqref{expancy}, we can  calculate its partial derivative 
\begin{align*}
\frac{\partial u_\varepsilon^s}{\partial z}(\m{r})&= -k a_{1,0} \sin(kz)+k b_{1,0}\sin(k(1-z))\\
&\quad-\sum_{(m,n)\neq(1,0)} \gamma_{mn} J_n(\alpha_{mn} r)( a_{mno}\sin(n\theta)+ a_{mne}\cos(n\theta) )\exp(-\gamma_{mn}z)\\
&\quad +\sum_{(m,n)\neq(1,0)} \gamma_{mn} J_n(\alpha_{mn} r)( b_{mno}\sin(n\theta)+b_{mne}\cos(n\theta) )\exp(-\gamma_{mn}(1-z)).
\end{align*}
Then the normal derivatives of $u_\varepsilon^s$ on $\Gamma_1$ and $\Gamma_2$ are given by
\begin{align}
\frac{\partial u_\varepsilon^s}{\partial z}(r,\theta,1)& = -k a_{1,0} \sin k-\sum_{(m,n)\neq(1,0)} \gamma_{mn} J_n(\alpha_{mn} r)( a_{mno}\sin(n\theta)+ a_{mne}\cos(n\theta) )\exp(-\gamma_{mn})\nonumber \\
&\quad+\sum_{(m,n)\neq(1,0)} \gamma_{mn} J_n(\alpha_{mn} r)( b_{mno}\sin(n\theta)+b_{mne}\cos(n\theta) ),\label{gamma1}\\
\frac{\partial u_\varepsilon^s}{\partial z}(r,\theta,0) &=k b_{1,0}\sin k-\sum_{(m,n)\neq(1,0)} \gamma_{mn}  J_n(\alpha_{mn} r)( a_{mno}\sin(n\theta)+ a_{mne}\cos(n\theta) )\nonumber\\
&\quad+\sum_{(m,n)\neq(1,0)} \gamma_{mn} J_n(\alpha_{mn} r)( b_{mno}\sin(n\theta)+b_{mne}\cos(n\theta) )\exp(-\gamma_{mn}).\label{gamma2}
\end{align}
By the change of variables as in Section~\ref{asym}, there hold for the scaled variable $\m{X}= (R,\theta)\in D(0,1)$ that
\begin{align*}
 \frac{\partial u_\varepsilon^s}{\partial z}(\varepsilon\m{X},1)= -\Psi_1(\m{X}), \\
 \frac{\partial u_\varepsilon^s}{\partial z}(\varepsilon\m{X},0)= \Psi_2(\m{X}), 
\end{align*}
 which, together with the expansions \eqref{gamma1int}--\eqref{gamma2int} and  \eqref{gamma1}--\eqref{gamma2}, lead to \begin{align*}
-k a_{1,0} \sin k &= -\langle \Psi_1, 1\rangle \nonumber \\
&=-(\alpha +d_1\cdot  O(k\varepsilon) + O(k^2\varepsilon^2)) \left(\frac{1}{p}+\frac{1}{q}\right), \\
kb_{1,0} \sin k &=\langle \Psi_2, 1\rangle \nonumber \\
&=(\alpha +d_1\cdot O(k\varepsilon) + O(k^2\varepsilon^2)) \left(\frac{1}{p}-\frac{1}{q}\right).
\end{align*}
Thus one obtain  approximations of $a_{1,0}$ and $b_{1,0}$,
\begin{align*}
a_{1,0} &= \frac{1}{k\sin k}(\alpha +d_1\cdot O(k\varepsilon) + O(k^2\varepsilon^2)) \left(\frac{1}{p}+\frac{1}{q}\right), \\
b_{1,0} &=\frac{1}{k\sin k}(\alpha +d_1\cdot O(k\varepsilon) + O(k^2\varepsilon^2)) \left(\frac{1}{p}-\frac{1}{q}\right).
\end{align*}

Next we estimate $\{a_{mne}, b_{mne}\}$. Representing the partial derivatives with  the formulas  \eqref{gamma1} and \eqref{gamma2}, we integrate  $\frac{\partial u_\varepsilon^s}{\partial z}(r,\theta,1)$ and $\frac{\partial u_\varepsilon^s}{\partial z}(r,\theta,0)$ with $J_n(\alpha_{mn} r)\cos(n\theta)$ over  $\Gamma_1$  and $\Gamma_2$ respectively:
\begin{align}
\gamma_{mn} \frac{1}{d_{mn}}(-a_{mne}\exp(-\gamma_{mn})+b_{mne})&= \int_{\Gamma_1} \frac{\partial u_\varepsilon^s}{\partial z}(r,\theta,1) J_n(\alpha_{mn} r) \cos(n\theta) ds(\m{r})\nonumber \\
&=-\varepsilon^2\langle \Psi_1, J_n(q_{mn} R) \cos(n\theta) \rangle, \label{expab}\\
\gamma_{mn} \frac{1}{d_{mn}}(-a_{mne}+b_{mne}\exp(-\gamma_{mn}))&= \int_{\Gamma_2} \frac{\partial u_\varepsilon^s}{\partial z}(r,\theta,0) J_n(\alpha_{mn} r) \cos(n\theta) ds(\m{r})\nonumber\\
&=-\varepsilon^2\langle \Psi_2, J_n(q_{mn} R) \cos(n\theta) \rangle,\label{expba}
\end{align}
where we have used the change of variables. From \eqref{expab} and \eqref{expba} we can deduce that $a_{mne}$, $b_{mne}$ satisfy
\begin{align}
\gamma_{mn}a_{mne} = \frac{d_{mn} \varepsilon^2}{\exp(-2\gamma_{mn})-1}(\exp(-\gamma_{mn})\langle \Psi_1, J_n(q_{mn} R) \cos(n\theta) \rangle-\langle \Psi_2, J_n(q_{mn} R) \cos(n\theta) \rangle),\label{amne}\\
\gamma_{mn}b_{mne} = \frac{d_{mn} \varepsilon^2}{\exp(-2\gamma_{mn})-1}(\exp(-\gamma_{mn})\langle \Psi_2, J_n(q_{mn} R) \cos(n\theta) \rangle-\langle \Psi_1, J_n(q_{mn} R) \cos(n\theta) \rangle).\label{bmne}
\end{align}
It follows from Lemma~\ref{enhancementpsi} that  
\begin{align*}
\Vert\Psi_1\Vert_{V_1}\lesssim \frac{1}{\varepsilon^2},\quad\Vert\Psi_2\Vert_{V_1}\lesssim \frac{1}{\varepsilon^2}.
\end{align*}
At the same time, from the definitions of Bessel functions, we have estimates
\begin{align*}
\Vert J_n(q_{mn} R) \cos(n\theta) \Vert_{L^2(D(0,1)}&\lesssim 1,\\
\Vert J_n(q_{mn} R) \cos(n\theta) \Vert_{H^1(D(0,1)}&\lesssim q_{mn},
\end{align*}
which, together with the Gagliardo-Nirenberg interpolation inequality,  yields
\begin{equation*}
\Vert J_n(q_{mn} R) \cos(n\theta) \Vert_{V_2}\lesssim \sqrt{q_{mn}}.  
\end{equation*}
Noting that $\gamma_{mn} = O(\frac{q_{mn}}{\varepsilon})$, $d_{mn}\varepsilon^2=O(1)$ for $(m,n)\neq(1,0)$,  we can conclude from \eqref{amne}--\eqref{bmne} with estimates above that 
\begin{equation*}
\sqrt{q_{mn}}|a_{mne}|\lesssim  O(\varepsilon^{-1}),\quad \sqrt{q_{mn}}|b_{mne}|\lesssim O(\varepsilon^{-1}),
\end{equation*}
which verifies \eqref{estqmn}. Following the  similar argument we could deduce the estimates for $\{a_{mno}, b_{mno}\}$.
\end{proof}
 With these expansions of coefficients in Lemma~\ref{a0b0}, we are ready to present the asymptotic expansions of $u_\varepsilon^s(\m{r})$ in $\Omega^\varepsilon$ in the following theorem.
\begin{theorem*}\label{intus}
The scattered field in the hole $\Omega^\varepsilon$ is given by
\begin{align*}
u_\varepsilon^s(\m{r}) &=\left(\frac{1}{\varepsilon^2} +d_1\cdot O(\varepsilon^{-1}) \right)\cdot\frac{2i}{(2n-1)\pi k\sin (k/2)}\cdot\cos(k(z-1/2))+\frac{\sin(k(z-1/2))}{\sin (k/2)}+O(\varepsilon),\\
u_\varepsilon^s(\m{r}) &=-\left(\frac{1}{\varepsilon^2} +d_1\cdot O(\varepsilon^{-1}) \right)\cdot \frac{i}{ n\pi k \cos(k/2)}\cdot\sin(k(z-1/2))+\frac{ \cos(k(z-1/2))}{\cos (k/2)}+O(\varepsilon),
\end{align*}
at the resonant frequencies $k = \operatorname{Re} k_{n,1}$ and $k=   \operatorname{Re} k_{n,2} $ respectively.
\end{theorem*}
\begin{proof}
Plugging the estimates of $a_{1,0}$, $b_{1,0}$, and $\{a_{mno},a_{mne}, b_{mno}, b_{mne}\}$ from Lemma~\ref{a0b0} into \eqref{expancy}, one has the expansion of $u_\varepsilon^s(\m{r})$ inside $\Omega^\varepsilon$ as 
\begin{align}
u_\varepsilon^s(\m{r})&=  a_{1,0} \cos(kz)+b_{1,0}\cos(k(1-z))+O(\varepsilon^{-1}\exp(-\varepsilon^{-1}))\nonumber \\
&=\frac{1}{k\sin k}\left(\alpha +d_1\cdot O(k\varepsilon) + O(k^2\varepsilon^2)\right)\nonumber\\
&\quad\cdot\left( \left(\frac{1}{p}+\frac{1}{q}\right)\cos(kz)+\left(\frac{1}{p}-\frac{1}{q}\right)\cos(k(1-z))\right)+O(\varepsilon^{-1}\exp(-\varepsilon^{-1}))\nonumber\\
&=\frac{1}{k\sin k}(\alpha +d_1\cdot O(k\varepsilon) + O(k^2\varepsilon^2))\nonumber\\
&\quad\cdot\left( (\cos(kz)+\cos(k(1-z)))\frac{1}{p}+(\cos(kz)-\cos(k(1-z)))\frac{1}{q}\right)+O(\varepsilon^{-1}\exp(-\varepsilon^{-1})).\label{uodd}
\end{align} 
It is noted that from Lemma~\ref{enhancepq}, at resonant frequency $k=\operatorname{Re} k_{n,1}$, 
\begin{equation}
\frac{1}{p}=\frac{2i}{(2n-1)\pi\alpha \varepsilon^2 }+O(\varepsilon^3) , \quad \frac{1}{q} = \frac{k\sin k}{(\cos k -1)\alpha}(1+O(\varepsilon)).\label{pqin}
\end{equation}
Then we plug \eqref{pqin} into \eqref{uodd} to deduce the representation of $u_\varepsilon^s(\m{r})$ at $k=\operatorname{Re}  k_{n,1}$:
\begin{align*}
u_\varepsilon^s(\m{r})& = \frac{1}{k\sin k}(1 +d_1\cdot O(k\varepsilon) + O(k^2\varepsilon^2))\cdot\bigg( 2\cos\left(k/2\right)\cos\left(k\left(z-1/2\right)\right)\frac{2i}{\varepsilon^2(2n-1)\pi}+O(\varepsilon^3) \\
&\quad-2\sin (k/2)\sin(k(z-1/2))\frac{k\sin k}{\cos k -1}(1+O(\varepsilon) )\bigg)+O(\varepsilon^{-1}\exp(-\varepsilon^{-1}))\\
&=(1 +d_1\cdot O(k\varepsilon) + O(k^2\varepsilon^2))\cdot\frac{2\cos(k(z-1/2))i}{\varepsilon^2(2n-1)\pi k\sin (k/2)}+O(\varepsilon^3)+\frac{\sin(k(z-1/2))}{\sin (k/2)}(1+O(\varepsilon)
)\\
&=(\frac{1}{\varepsilon^2} +d_1\cdot O(\frac{k}{\varepsilon}) + O(k^2))\cdot\frac{2\cos(k(z-1/2))i}{(2n-1)\pi k\sin (k/2)}+\frac{\sin(k(z-1/2))}{\sin (k/2)}+O(\varepsilon).
\end{align*}
Similarly at even frequencies $k=\operatorname{Re}  k_{n,2}$, one obtain 
  \begin{equation}
\frac{1}{p} =\frac{k\sin k}{\alpha(\cos k +1)}(1+O(\varepsilon)), \quad \frac{1}{q}=\frac{i}{\alpha n\pi \varepsilon^2}+O(\varepsilon^3).
 \end{equation}

Then we  have the representation of $u_\varepsilon^s(\m{r})$ as
\begin{align*}
u_\varepsilon^s(\m{r})&=\frac{1}{k\sin k}(\alpha +d_1\cdot O(k\varepsilon) + O(k^2\varepsilon^2))\nonumber\\
&\quad\cdot\left(  2\cos(k/2)\cos(k(z-1/2))\frac{1}{p}- 2\sin(k/2)\sin(k(z-1/2))\frac{1}{q}\right)+O(\varepsilon^{-1}\exp(-\varepsilon^{-1}))\\
&=(1 +d_1\cdot O(k\varepsilon) + O(k^2\varepsilon^2))\nonumber\\
&\quad \cdot\left( \left(\frac{ \cos(k(z-1/2))}{\cos (k/2)}(1+O(\varepsilon))\right)+\left(\frac{-i\sin(k(z-1/2))}{k\cos(k/2)  n \pi\varepsilon^2}+O(\varepsilon^3)\right)\right)+O(\varepsilon^{-1}\exp(-\varepsilon^{-1}))\\
&=-\left(\frac{1}{\varepsilon^2} +d_1\cdot O(k\varepsilon^{-1}) + O(k^2)\right)\cdot\frac{\sin(k(z-1/2))i}{ n\pi k \cos(k/2)}+\frac{ \cos(k(z-1/2))}{\cos (k/2)}+O(\varepsilon).
\end{align*}
\end{proof}
It is shown in Theorem~\ref{intus} that the enhancement inside the hole at resonant frequencies is of order $O(\varepsilon^{-2})$.  Moreover,  we can  observe from Theorem~\ref{intus} that the dominant resonant modes are simply $\cos(k(z-1/2))$ and $\sin(k(z-1/2))$ in the cylindrical hole at the odd and even resonant frequencies respectively.

\end{subsection}

\end{section}
\begin{section}{Quantitative analysis of the field enhancement in the nonresonant quasi-static regime }\label{quantnon}
Except for the resonant frequencies, we are also interested in the asymptotic behavior of the scattered field in the quasi-static regime, i.e., when the wavenumber $k \ll 1$. In \cite{lin2015electromagnetic} it is reported that for the electromagnetic scattering, there exists strong enhancement of the electric field inside the subwavelength slit in the quasi-static regime, which motivates us to carry out similar investigations on the acoustic field. 
 As shown in Theorem~\ref{thmreso}, there exists no Fabry-Perot type  resonance when $k\ll 1$. And according to the formula \eqref{farexpu}, no field enhancement of the scattered pressure field will be observed in the far field. Thus we do not elaborate the investigation on scattered wave in the far field, and turn to the asymptotic expansion of the scattered field $u_\varepsilon^s(\m{r})$ inside of the circular hole $\Omega^\varepsilon$ in the quasi-static regime. 

Following the formula \eqref{expancy}, we obtain the following expansion of the scattered field inside the hole: 
\begin{align}
u_\varepsilon^s(\m{r})&:=  u_0(\m{r})+ u_\infty\nonumber \\
&=(\alpha +d_1\cdot  O(k\varepsilon) + O(k^2\varepsilon^2))\cdot\left(\frac{\cos (kz)}{k \sin k}\left(\frac{1}{p}+\frac{1}{q}\right)+\frac{\cos (k(1-z))}{k \sin k}\left(\frac{1}{p}-\frac{1}{q}\right)\right)+u_\infty, \label{suinside}
\end{align}
where $u_\infty = O(\varepsilon^{-1}\exp(-\varepsilon^{-1}))$.
If $k\ll 1$, from the Taylor' expansions over $k$ of the formulas \eqref{expofp}--\eqref{expofq} of $p$, $q$,
we can derive 
\begin{align*}
\frac{1}{k \sin k}\frac{1}{p}= \frac{1}{(\cos k +1)\alpha}(1+O(k^2\varepsilon^2)), \\
\frac{1}{k \sin k}\frac{1}{q}= \frac{1}{(\cos k -1)\alpha}(1+O(k^2\varepsilon^2)). 
\end{align*}
Plugging these expansions above into \eqref{suinside}, one can deduce that   
\begin{align*}
u_0(\m{r})&= (\alpha +d_1\cdot  O(k\varepsilon) + O(k^2\varepsilon^2)) \left(\frac{\cos (kz)}{k \sin k}\left(\frac{1}{p}+\frac{1}{q}\right)+\frac{\cos (k(1-z))}{k \sin k}\left(\frac{1}{p}-\frac{1}{q}\right)\right)\\
&= (1 + d_1\cdot O(k\varepsilon) + O(k^2\varepsilon^2))(1+O(k^2\varepsilon^2)) \\
&\quad\cdot\left(\cos (kz)\left(\frac{1}{\cos k +1}+\frac{1}{\cos k -1}\right)+\cos (k(1-z))\left(\frac{1}{\cos k +1}-\frac{1}{\cos k -1}\right)\right)\\
&=(1 + d_1\cdot  O(k\varepsilon) + O(k^2\varepsilon^2))\frac{2\sin(kz)  }{\sin k}.
\end{align*}
From the asymptotic expansion over $k$ of this formula, one obtain
\begin{equation}
u_0(\m{r})=2z+ O(k^2)+ d_1\cdot O(k\varepsilon) + O(k^2\varepsilon^2),\label{prequasi}
\end{equation}  
which implies that the scattered pressure field inside the cylindrical hole is not enhanced when $k\ll 1$ and $\varepsilon\ll 1$.

 Although the scattered pressure field $u_\varepsilon^s(\m{r})$ is not enhanced compared to the incident pressure field,
it is noted that the leading-order term $u_0(\m{r})$ of $u_\varepsilon^s(\m{r})$ has constant gradient along $z$ axis, and we shall investigate the asymptotic behavior of the velocity field $\m{p}^s(\m{r}): = \frac{1}{k}\nabla u_\varepsilon^s(\m{r})$.
\begin{theorem*}\label{enhancequasihole}
When $k\ll 1$, $\varepsilon\ll 1$, the scattered velocity field $\m{p}^s(\m{r})=(p_1^s, p_2^s, p_3^s)$ satisfies:
\begin{align*}
p^s_1(\m{r})&\sim O(k)+  O(\varepsilon)+O(k^{-1}\varepsilon^{-1}\exp(-\varepsilon^{-1})),\\
p^s_2(\m{r})&\sim O(k)+  O(\varepsilon)+O(k^{-1}\varepsilon^{-1}\exp(-\varepsilon^{-1})),\\
p^s_3(\m{r})&= \frac{2}{k}+O(k)+d_1\cdot O(\varepsilon)+O(k^{-1}\varepsilon^{-1}\exp(-\varepsilon^{-1})).
\end{align*}
\end{theorem*}
\begin{proof}
From  the expansion \eqref{prequasi} of $u_0(\m{r})$ for  $\m{r}\in\Omega^\varepsilon$, using the Cartesian coordinate system $\m{r} = (x,y,z)$, one can calculate by definitions that:
\begin{align*}
p^s_1(\m{r})&= \dfrac{1}{k}\dfrac{\partial u_0}{\partial x}(\m{r})+ \dfrac{1}{k}\dfrac{\partial u_\infty}{\partial x}(\m{r}) = O(k)+ d_1\cdot O(\varepsilon)+O(k^{-1}\varepsilon^{-1}\exp(-\varepsilon^{-1})),\\
p^s_2(\m{r})&=\dfrac{1}{k} \dfrac{\partial u_0}{\partial x}(\m{r})+ \dfrac{1}{k}\dfrac{\partial u_\infty}{\partial y}(\m{r}) = O(k)+d_1\cdot O(\varepsilon)+O(k^{-1}\varepsilon^{-1}\exp(-\varepsilon^{-1})),\\
p^s_3(\m{r})&=\dfrac{1}{k}\dfrac{\partial u_0}{\partial x}(\m{r})++ \dfrac{1}{k}\dfrac{\partial u_\infty}{\partial z}(\m{r}) = \dfrac{2}{k}+O(k)+d_1\cdot O(\varepsilon)+O(k^{-1}\varepsilon^{-1}\exp(-\varepsilon^{-1})),
\end{align*}
and the desired conclusion follows.
\end{proof}
 The expansions in Theorem~\ref{enhancequasihole} for $\ell=1$ can be generated to the case when $\ell\neq 1$ by a scaling argument and concluded as the following proposition.

\begin{proposition*}
If $k\ll 1$ and $\varepsilon\ll \min\{\ell,1\}$, the scattered velocity field $\m{p}^s(\m{r})=(p_1^s, p_2^s, p_3^s)$ satisfies:
\begin{align*}
p^s_1(\m{r})&\sim O(k\ell)+  O(\varepsilon)+O((k \ell )^{-1}\varepsilon^{-1}\exp(-\varepsilon^{-1})),\\
p^s_2(\m{r})&\sim O(k\ell)+  O(\varepsilon)+O((k\ell)^{-1}\varepsilon^{-1}\exp(-\varepsilon^{-1})),\\
p^s_3(\m{r})&= \frac{2}{k\ell}+O(k\ell)+d_1\cdot O(\varepsilon)+O((k\ell)^{-1}\varepsilon^{-1}\exp(-\varepsilon^{-1})).
\end{align*}
\end{proposition*}

The corresponding velocity fields to the incident pressure field $u^i(\m{r}) = e^{ik (d_1 r\cos\theta -d_3(z-\ell))}$ and  the reflected pressure  field $u^r(\m{r}) = e^{ik (d_1 r\cos\theta+d_3(z-\ell))}$ are
\begin{equation}
\m{p}^i(\m{r}) = e^{ik (d_1 r\cos\theta -d_3(z-\ell))}(d_1,0,d_3) i=O(1), \quad \m{p}^r(\m{r}) = e^{ik (d_1 r\cos\theta+d_3(z-\ell))}(d_1, 0, -d_3) i=O(1).
\end{equation}

Thus we conclude that compared to the incident and reflected velocity fields, the velocity field inside the hole is enhanced with an enhancement order of $O(\frac{1}{k\ell})$ in the nonresonant quasi-static regime due to the fast transition of the pressure field from the upper to the lower circular aperture.

\end{section}
\begin{section}{Conclusion}
In this work we have studied the mechanism of the field enhancement caused by a circular hole perforated in an infinite plate at the resonant frequencies and in the quasi-static regime. We have derived the Green's function of each subdomain to build the boundary-integral equations, and carried out asymptotic expansions to rigorously prove the existence of Fabry-Perot type resonances. We have derived the asymptotic expansions for those resonances, and  quantitatively analyzed the field enhancement near these resonant frequencies.  The asymptotic behavior of fields at nonresonant frequencies in the quasi-static regime has also been investigated. 
 We have proved that the enhancement of the pressure field with an order of $O(\varepsilon^{-2})$ occurs at the Fabry-Perot type resonances, and the enhancement of the velocity field with an order of $O((k\ell)^{-1})$ occurs at the nonresonant frequencies in the quasi-static regime.
 The study of the single aperture could be extended to the field enhancement of other subwavelength structures that use the single aperture as their building block.

\end{section}
\section*{Acknowledgement}
The authors wish to thank Prof. Hai Zhang from HKUST for his many  insightful suggestions and helpful discussions which make this work more comprehensive.

\bibliographystyle{unsrt}

\begin{thebibliography}{1}
\bibitem{ebbesen1998extraordinary}
 T. W. Ebbesen, H. J. Lezec,  H. F. Ghaemi, T. Thio, P. A.  Wolff, Extraordinary optical transmission through sub-wavelength hole arrays, Nature, 391.6668 (1998), pp. 667-669.


\bibitem{wu2018perspective}
 Y. Wu,  M.  Yang, P. Sheng, Perspective: Acoustic metamaterials in transition,  J. Appl. Phys., 123.9(2018), pp. 090901.


\bibitem{yang2017sound}
 M. Yang, P. Sheng, Sound absorption structures: From porous media to acoustic metamaterials, Ann. Rev. Mater. Res., 47(2017), pp. 83-114.


\bibitem{kress1989linear}
 R. Kress, V. Maz'ya, V. Kozlov, Linear integral equations (Vol. 82), Berlin: Springer, 1989. 

%


\bibitem{miles2002resonant}
J. Miles, On resonant reflection by a plane grating, Wave motion, 35.4 (2002), pp. 311-314.

\bibitem{seo2000acoustic}
 J. S. Seo, H. J. Eom, and H. S. Lee, Acoustic scattering from two circular apertures in a thick hard plane, J. Acoust. Soc. Am.,  107.5 (2000), pp. 2338-2343.
 
 \bibitem{ammari2002analysis}
 H. Ammari, G. Bao, A. W. Wood, Analysis of the electromagnetic scattering from a cavity, Jpn. J. Ind. Appl. Math.,  19.2(2002), pp.301-310.


\bibitem{zhang2005acoustic}
 X. Zhang, Acoustic resonant transmission through acoustic gratings with very narrow slits: Multiple-scattering numerical simulations,   Phys. Rev. B, 71.24 (2005), pp. 241102.


\bibitem{lin2017scattering}
J. Lin,  H. Zhang, Scattering and field enhancement of a perfect conducting narrow slit,  SIAM J. Appl. Math., 77.3 (2017), pp. 951-976.


\bibitem{lin2018scattering}
J. Lin, H. Zhang,  Scattering by a periodic array of subwavelength slits I: field enhancement in the diffraction regime,  Multiscale Model Sim., 16.2(2018), PP. 922-953.

\bibitem{ammari2017sub}
 H. Ammari,  B. Fitzpatrick, D. Gontier,  H. Lee, H. Zhang, Sub-wavelength focusing of acoustic waves in bubbly media, P Roy. Soc. A-Math. Phy. , 473.2208(2017), pp. 20170469.


\bibitem{ammari2015mathematical}
H. Ammari, H. Zhang, A mathematical theory of super-resolution by using a system of sub-wavelength Helmholtz resonators,  Commun. Math. Phys.,  337.1 (2015), pp. 379-428.

\bibitem{colton2012inverse}
D. Colton,  R. Kress, Inverse acoustic and electromagnetic scattering theory (Vol. 93),  Springer Science \& Business Media, 2012.



\bibitem{tai1994dyadic}
C. T. Tai, Dyadic Green functions in electromagnetic theory, IEEE press, 1994.


\bibitem{collin1990field}
R. E. Collin,  Field theory of guided waves (Vol. 5). John Wiley \& Sons, 1990.
  
  
\bibitem{adams2003sobolev}
R. A. Adams, J. J. Fournier,  Sobolev spaces, Elsevier, 2003.

\bibitem{gradshteyn2014table}
I. S. Gradshteyn , I. M. Ryzhik, Table of integrals, series, and products, Academic press, 2014.

\bibitem{lezec2002beaming}
 H. J. Lezec,  A. Degiron,  E. Devaux,  R. A. Linke, L. Martin-Moreno,  F. J. Garcia-Vidal, T. W. Ebbesen, Beaming light from a subwavelength aperture, Science, 297.5582 (2002), pp. 820-822.

\bibitem{martin2004optical}
 L. Martin-Moreno, F. J. Garcia-Vidal, Optical transmission through circular hole arrays in optically thick metal films, Opt. Express, 12.16 (2004), pp. 3619-3628.

\bibitem{astilean2000light}
S. Astilean, Ph. Lalanne, and M. Palamaru, Light transmission through metallic channels much smaller than the wavelength, Opt. Commun., 175.4-6 (2000), pp. 265-273.

\bibitem{bonnetier2010asymptotic}
 E. Bonnetier, F. Triki, Asymptotic of the Green function for the diffraction by a perfectly conducting plane perturbed by a sub-wavelength rectangular cavity,  Math. Method Appl. Sci., 33.6 (2010), pp. 772-798.

\bibitem{Ori2020}
 R. Brandao, O. Schnitzer, Acoustic impedance of a cylindrical orifice, J. Fluid. Mech., 892(2020), pp. A7.
 \bibitem{holley2019extraordinary}
 J.R. Holley, O. Schnitzer, Extraordinary transmission through a narrow slit, Wave Motion, 91(2019), pp. 102381.



\bibitem{ammari2009layer}
H. Ammari, H. Kang,  H. Lee, Layer potential techniques in spectral analysis, Math. Surveys Monographs (Vol. 153), AMS, Providence, RI, 2009.

\bibitem{hou2013enhanced}
 B. Hou,  W. Wen, Enhanced Transmission of acoustic waves through subwavelength holes in hard plates, Modeling and Measurement Methods for Acoustic Waves and for Acoustic Microdevices (Vol. 88), BoD-Books on Demand, 2013.

\bibitem{christensen2008theory}
J. Christensen, L. Martin-Moreno, F.J. Garcia-Vidal, Theory of resonant acoustic transmission through subwavelength apertures, Phys. Rev. Lett., 101.1(2008), pp. 014301.

\bibitem{lin2015electromagnetic}
J. Lin, F. Reitich, Electromagnetic field enhancement in small gaps: a rigorous mathematical theory, SIAM J.  Appl. Math.,  75.5(2015), pp. 2290-2310.
\end{thebibliography}

\end{document}